\documentclass[a4paper,reqno]{amsart}

\usepackage[utf8]{inputenc}
\usepackage{xcolor}
\usepackage[colorlinks,
    linkcolor={red!50!black},
    citecolor={blue!50!black},
    urlcolor={blue!80!black}]{hyperref}

\usepackage{amsfonts, amstext, amsmath, amsthm, amscd, amssymb,enumitem}
\usepackage{graphicx, comment,setspace,pinlabel,caption,comment,url}

\usepackage[all]{xy}
\xyoption{import,dvips,rotate}

\DeclareMathOperator{\Hom}{Hom}

\DeclareMathOperator{\interior}{int}

\DeclareMathOperator{\rk}{rk}
\DeclareMathOperator{\lk}{lk}
\DeclareMathOperator{\pt}{pt}
\DeclareMathOperator{\ord}{ord}
\DeclareMathOperator{\Char}{Char}

\renewcommand{\emptyset}{\varnothing}
\renewcommand{\phi}{\varphi}
\newcommand{\eps}{\varepsilon}

\def\B{\mathcal{B}}

\def\Z{\mathbb{Z}}
\def\N{\mathbb{N}}
\def\Q{\mathbb{Q}}

\def\lmat{\left(\begin{smallmatrix}}
\def\rmat{\end{smallmatrix}\right)}

\def\sm{\setminus}
\def\wti{\widetilde}

\def\id{\operatorname{id}}

\newcommand{\tmfrac}[2]{\mbox{\large$\frac{#1}{#2}$}} 

\theoremstyle{plain}
\newtheorem{theorem}{Theorem}[section]
\newtheorem{proposition}[theorem]{Proposition}
\newtheorem{lemma}[theorem]{Lemma}
\newtheorem{corollary}[theorem]{Corollary}
\newtheorem{conjecture}[theorem]{Conjecture}
\newtheorem{question}[theorem]{Question}

\theoremstyle{definition}
\newtheorem{definition}[theorem]{Definition}

\theoremstyle{remark}

\newtheorem*{remark}{Remark}
\newtheorem*{claim}{Claim}

\def\ba{\begin{array}}
\def\ea{\end{array}}
\def\ol{\overline}
\def\bn{\begin{enumerate}}
\def\en{\end{enumerate}}

\newcounter{myenum1}
  {\end{list}}

\newenvironment{flushenumerate(i)}{
\begin{enumerate}[(i)]
  \setlength{\leftmargin}{0pt}
}{\end{enumerate}}

\begin{document}
\title[Satellites and concordance of knots in 3--manifolds]{Satellites and concordance of knots in 3--manifolds}

\author{Stefan Friedl}
\address{Fakult\"at f\"ur Mathematik\\ Universit\"at Regensburg\\   Germany}
\email{stefan.friedl@mathematik.uni-regensburg.de}

\author{Matthias Nagel}
\address{D\'epartement de Math\'ematiques,
Universit\'e du Qu\'ebec \`a Montr\'eal, Canada}
\email{nagel@cirget.ca}

\author{Patrick Orson}
\address{D\'epartement de Math\'ematiques,
Universit\'e du Qu\'ebec \`a Montr\'eal, Canada}
\email{patrick.orson@cirget.ca}

\author{Mark Powell}
\address{D\'epartement de Math\'ematiques,
Universit\'e du Qu\'ebec \`a Montr\'eal, Canada}
\email{mark@cirget.ca}

\date{\today}

\def\subjclassname{\textup{2010} Mathematics Subject Classification}
\expandafter\let\csname subjclassname@1991\endcsname=\subjclassname
\expandafter\let\csname subjclassname@2000\endcsname=\subjclassname

\begin{abstract}
Given a $3$--manifold $Y$ and a free homotopy class in $[S^1,Y]$, we
investigate the set of topological concordance classes
of knots in $Y \times [0,1]$ representing the given homotopy class.
The concordance group of knots in the $3$--sphere acts on
this set. We show in many cases that the action is not transitive,
using two techniques.
Our first technique uses Reidemeister torsion invariants, and the second uses linking
numbers in covering spaces.  In particular, we show using covering links that
for the trivial homotopy class, and for any $3$--manifold
that is not the $3$--sphere, the set of orbits is infinite.  On the
other hand, for the case that $Y=S^1 \times S^2$, we apply
topological surgery theory to show that all knots with winding number one
are concordant.
\end{abstract}
\maketitle

\section{Introduction}

In this paper we study the problem of concordance of knots in general 3--manifolds. Throughout the paper, embeddings are assumed to be locally flat unless specified otherwise. Let $Y^3$ be a closed, oriented 3--manifold and fix an orientation for $S^1$. An embedding $L\colon \bigsqcup_mS^1\hookrightarrow Y^3$, considered up to ambient isotopy, is called an \emph{$m$--component link} and a 1--component link is called a \emph{knot}. We will sometimes write $L\subset Y$ as a shorthand for a link. Links $L_0$ and $L_1$ are \emph{concordant} if there exists a proper embedding $A\colon \bigsqcup_mS^1\times[0,1]\hookrightarrow Y\times[0,1]$ such that $L_0=A|_{\bigsqcup_mS^1\times\{0\}}\subset Y\times\{0\}$ and $L_1=A|_{\bigsqcup_mS^1\times\{1\}}\subset Y\times\{1\}$, in which case we say that $A$ is a \emph{concordance} between the links.

Denote the equivalence relation of concordance by $L_0\sim L_1$ and the set of concordance equivalence classes of knots in $Y$ by $\mathcal{C}(Y)$. For $Y=S^3$ we write $\mathcal{C}=\mathcal{C}(S^3)$. For topological spaces $U,V$, denote the set of \emph{free} (i.e.\ unbased) homotopy classes of maps $U\to V$ by $[U,V]$. The composition of any continuous function $A$ with projection to $Y$\[S^1\times[0,1]\xrightarrow{A} Y\times[0,1]\to Y\]is a continuous function, so concordant knots have the same unbased homotopy class.
We denote  the set of concordance classes of knots which
realise a given unbased homotopy class $x\in[S^1,Y]$ by $\mathcal{C}_x(Y)$, and we observe there is a partition of sets
\[\mathcal{C}(Y)=\bigsqcup_{x\in[S^1,Y]}\mathcal{C}_x(Y).\]
To study concordance of knots in 3--manifolds, we will fix a pair $(Y,x)$ with
$x\in[S^1,Y]$, and investigate the set $\mathcal{C}_x(Y)$.

\subsection{Almost-concordance}

The connected sum of knots $(S^3,J)\# (Y,K)$ defines a new knot $(Y,J\#K)$, which is freely homotopic to $K$ in~$Y$, since all knots in $S^3$ are freely null-homotopic.

\begin{definition}\label{def:local}
For each pair $(Y,x)$, the \emph{local action} of the concordance group~$\mathcal{C}$
on the set~$\mathcal{C}_x(Y)$ is
defined by
\[\mathcal{C}\times\mathcal{C}_x(Y)\to \mathcal{C}_x(Y)\qquad
([J],[K])\mapsto [J\# K].\]For $K\in \mathcal{C}_x(Y)$ the \emph{almost-concordance} class of $K$ is the orbit of $K$ under the local action of
$\mathcal{C}$ on $\mathcal{C}_x(Y)$.
\end{definition}

The study of this local action can be traced back to Milnor's study of link homotopy using what are now called the Milnor's $\overline{\mu}$ invariants~\cite{Milnor:1957-1}. These invariants of a link $L\subset S^3$ come from looking at quotients of $\pi_1(S^3\sm L)$ by the lower central subgroups, which has the intentional effect of making them blind to local knotting. Extensions of these invariants to \emph{knots} in general $3$--manifolds were studied by Miller~\cite{MR1354382}, Schneiderman~\cite{MR2012959} and Heck~\cite{Heck:2011}. In particular, Miller produces an infinite family of knots in $Y=S^1\times S^1\times S^1$, each homotopic to $S^1\times\{\pt\}\times\{\pt\}$, and his invariants can be used to show that these knots are mutually distinct in almost-concordance.

Celoria~\cite{Celoria:2016jk}, who coined the term `almost-concordance', recently studied the case of the null-homotopic class in lens spaces $L(n,1)$ with $n\geq3$. Using a generalisation of the $\tau$ invariant from knot Floer homology, he showed the existence of an infinite family of knots, mutually distinct in smooth almost-concordance. In \cite[Conjecture 44]{Celoria:2016jk} it is conjectured that,
when $Y\neq S^3$, within each free homotopy class in $[S^1,Y]$ there are
infinitely many distinct almost-concordance classes. For $Y=S^3$ the local action is transitive.

We prove in the first section of this paper that, besides $S^3$, there is another much less obvious case which must be excluded from such a conjecture. In the case~$Y = S^1 \times S^2$ and $x \in [S^1, Y]$ a primitive element, we determine that the set~$\mathcal{C}_x(Y)$ consists of a single class. The proof of this fact uses topological surgery theory, see Theorem~\ref{thm:topsurgery}.

\begin{theorem}[Concordance light bulb theorem]
If a knot $K$ in $S^1\times S^2$ is freely homotopic to $S^1\times\{\pt\}$,
then $K$ is concordant to $S^1\times\{\pt\}$.
\end{theorem}

So the following adjusted conjecture is a good starting point for studying almost-concordance, and is the central focus of this paper.

\begin{conjecture}\label{conj:main}
Fix a closed 3--manifold $Y$ and a free homotopy class $x\in [S^1,Y]$. Unless
$Y=S^3$, or $(Y,x)=(S^1\times S^2,[S^1\times\{\pt\}])$, there are infinitely
many distinct almost-concordance classes within the set $\mathcal{C}_x(Y)$.
\end{conjecture}

In this paper we have two main sets of results towards proving Conjecture \ref{conj:main}. The results employ very different techniques and are each effective under different circumstances.

As our first main result, we prove the following statement in Theorem~\ref{thm:NonTorsionCase}, which generates many infinite families of examples confirming Conjecture \ref{conj:main}.

\begin{theorem}
Let $Y$ be a closed, orientable $3$--manifold and $x \in [S^1, Y]$ a free homotopy class. Denote its homology
class with $[x] \in H_1(Y; \Z)$. If $[x] = 2u$ for a primitive class~$u$ of
infinite order, then there are infinitely many distinct almost-concordance
classes within the set $\mathcal{C}_x(Y)$.
\end{theorem}

We expect that one could push the methods of this paper further, using suitably cunning calculations, to deal with the case that $[x]=nu$  for a primitive class~$u$ of
infinite order and any $n\geq 2$. The case $n=1$ must be excluded as a consequence of Theorem \ref{thm:topsurgery}.

The theorem is proved using twisted Reidemeister torsion combined with a satellite construction.
The actual topological almost-concordance
invariants we obtain are rather technical to state, so we leave their precise
formulation to the body of the paper in Corollary~\ref{cor:TorsionInvariant}.
The principle is roughly as follows.
The local action of $\mathcal{C}$ can only affect the twisted
Reidemeister torsion of a knot $K$ in a 3--manifold by multiplication with the
Alexander polynomial of a knot in $S^3$, where the variable of the polynomial
corresponds to the meridian of the knot $K\subset Y$. If we can modify $K$ in
$Y$ so that we introduce more drastic changes to the twisted torsion, but do
not change the free homotopy class of the knot, then we can potentially change the
almost-concordance class in a detectable way. The technical difficulty we have
overcome is in choosing free homotopy classes of knots, and a coefficient system
for the homology, so that the twisted Reidemeister torsion is both non-trivial
and a concordance invariant.

Our second main technique for distinguishing almost-concordance classes comes from analysing linking numbers of covering links. In Section \ref{null-homotopic-case} we give a self-contained and elementary proof of the following.

\begin{proposition}
Let $Y$ be a spherical space form and let $x$ be the null-homotopic free homotopy class.  Then if $Y\neq S^3$, the set $\mathcal{C}_x(Y)$ contains infinitely many almost-concordance classes.
\end{proposition}

The proof exhibits knots that lift to links in $S^3$ with different pairwise linking numbers. We observe that almost-concordant knots lift to links for which the respective sets of pairwise linking numbers coincide, and the result follows. We remark that the examples exhibited by Celoria \cite{Celoria:2016jk} also lift to links in $S^3$ with different pairwise linking numbers. So in particular these examples are distinct in topological almost-concordance, a fact which cannot be detected by the $\tau$ invariant.

We proceed to generalise this idea into a more systematic approach, which can be applied to any 3--manifold $Y\neq S^3$ to obtain the second main theorem of this paper.

\begin{theorem}
For any closed orientable 3--manifold $Y \neq S^3$ and $x$ the null-homotopic class there are infinitely many distinct almost-concordance classes within the set~$\mathcal{C}_x(Y)$.
\end{theorem}

We note that it seems likely this theorem could also be proved using Schneiderman's concordance invariant from~\cite{MR2012959}.

In fact, the covering link obstruction we develop in the final section of the paper works even more generally, in the case where $x \in [S^1, Y]$ is \emph{any} torsion class. So even stronger versions of the above statement are proven in Corollary~\ref{cor:TorstionCase} and Theorem~\ref{thm:non-separating-surface}.

\subsection{Almost-concordance and piecewise linear $I$--equivalence}

The roots of almost-concordance in the smooth category go back to the 1960s. Stallings \cite{zbMATH03218889} defined links $L_0$ and $L_1$ to be \emph{$I$--equivalent} if there exists a proper (not necessarily locally flat) embedding $A\colon \bigsqcup_mS^1\times[0,1]\hookrightarrow Y\times[0,1]$ such that $L_0=A|_{\bigsqcup_mS^1\times\{0\}}\subset Y\times\{0\}$ and $L_1=A|_{\bigsqcup_mS^1\times\{1\}}\subset Y\times\{1\}$. Observe that a concordance is then a locally flat $I$--equivalence and a smooth concordance is a smooth $I$--equivalence. The intermediate notion of a piecewise linear $I$--equivalence turns out to be highly relevant to our current discussion. Precisely, given a closed 3--manifold $Y$ it is the main result of Rolfsen \cite{zbMATH03917275} that the smooth almost-concordance class of a knot is exactly the PL $I$--equivalence class of the knot. There does not appear to be a similar interpretation of topological almost-concordance directly in terms of some kind of $I$--equivalence. For discussions of classical invariants of PL $I$--equivalence see Rolfsen \cite{zbMATH03917275} and Hillman \cite[\textsection 1.5]{MR2931688}, also compare Goldsmith \cite{MR521732}.

\subsection{Almost-concordance and homology surgery}

Cappell and Shaneson~\cite{MR0339216} developed powerful tools to analyse knot concordance via homology surgery. We now briefly discuss how the subtleties of their setup relate to almost-concordance and to our invariants.

To apply their method to study the concordance set $\mathcal{C}_x(Y)$, one must first pick a knot $K$ representing $x$ as a target knot. A knot $J\subset Y$ is then called \emph{$K$--characteristic} if there exists a degree 1 normal map (of pairs) from the exterior of $J$ to the exterior of $K$, which is the identity on the boundary. A concordance from $J$ to $J'$ is called \emph{$K$--characteristic} if there exists a degree 1 normal map (of triads) from the exterior of the concordance to the exterior of the trivial concordance from $K$ to itself, restricting to the identity on the interior boundary and restricting to degree 1 normal maps on the respective exteriors of $J$ and $J'$. Write $\Char_K(Y)$ for the set of equivalence classes of $K$--characteristic knots in $Y$ modulo $K$--characteristic concordance.

If $Y=S^3$ and $U$ is the unknot, every knot and concordance is $U$--characteristic in a natural way. Local knotting defines an action of the group $\Char_U(S^3)\cong\mathcal{C}$ on $\Char_K(Y)$, which intertwines with the action of $\mathcal{C}$ on $\mathcal{C}_x(Y)$ under the natural map $i_{K,Y}\colon \Char_K(Y)\to \mathcal{C}_x(Y)$.
Consequently, almost-concordance invariants are invariant on the orbits of the action of $\mathcal{C}$ on $\Char_K(Y)$, and could potentially also determine whether a knot is $K$--characteristic. However, the maps $i_{K,Y}$ are not known to be injective in general, so statements in the reverse direction are less clear. We observe that the question of injectivity of these maps is closely analogous to the question of whether slice boundary links in $S^3$ are moreover boundary slice.

However, in practice, historical examples of characteristic concordance obstructions turn out to be almost-concordance invariants. In particular, the invariants of Miller~\cite{MR1354382}, mentioned above, which were not originally intended as almost-concordance invariants, turn out to be so. We note that Goldsmith~\cite{MR521732} has related Milnor's invariants for classical links to linking numbers in covering links. This suggests that our covering link invariants may be closely related to Miller's knot invariants.

To place our two approaches in the context of Cappell-Shaneson's theory, our covering link invariants could also be used to obstruct a knot being characteristic, whereas our Reidemeister torsion invariants can be thought of as detecting non-triviality in a quotient of Cappell-Shaneson homology surgery obstruction groups $\Gamma_4(\Z[\Z\times\Z/2\Z]\to\Z)/\Gamma_4(\Z[\Z]\to \Z)$.

\subsection{Further questions}

As the study of PL $I$--equivalence was largely conducted before the seminal work of Freedman, and so before the current appreciation of the difference between smooth and topological concordance, the aforementioned classical PL $I$--equivalence invariants are really invariants of topological almost-concordance. This suggests the following question, which has not been classically studied (and which we do not address in this present work).

\begin{question}\label{q:cat}
Are there knots $K, K'\subset Y$ which are topologically but not smoothly almost-concordant to one another?
\end{question}

This paper is focussed on the whether the local action of Definition~\ref{def:local} is transitive, but one can ask about other properties of this action.

\begin{question}
Given a $3$--manifold $Y$ and class $x\in [S^1,Y]$, when is the local action $\mathcal{C}\times\mathcal{C}_x(Y)\to \mathcal{C}_x(Y)$ free? When is it faithful?
\end{question}

For the question of whether this action is free, compare with the action considered in a high-dimensional setting by Cappell--Shaneson~\cite[\textsection 6]{MR0339216}.

\subsection*{Acknowledgments}
We thank Duncan McCoy for helpful discussions.
We are indebted to the referee for highlighting additional connections
to prior work and for further helpful comments.
SF was supported by the SFB 1085 `Higher Invariants' at the University of Regensburg, funded by the Deutsche Forschungsgemeinschaft (DFG). SF is grateful for the hospitality received at the University of Durham, and wishes to thank Wolfgang L\"uck for supporting a long stay at the Hausdorff Institute. MN is supported by a CIRGET postdoctoral fellowship. PO was supported by the EPSRC grant EP/M000389/1 of Andrew Lobb. MP was supported by an NSERC Discovery Grant. MN, PO and MP all thank the Hausdorff Institute for Mathematics in Bonn for both support and its outstanding research environment.


\section{A case when almost-concordance is trivial}\label{sec:topsurgery}

For $Z$ a submanifold of a manifold $X$, we introduce the notation $\nu
Z\subset X$ for some open tubular neighbourhood of $Z$ in $X$.

In this section, let $Y=S^1 \times S^2$ and let $x$ be the free homotopy class
of a generator $J:= S^1 \times \{\pt\} \subset S^1 \times S^2$ of
$\pi_1(S^1 \times S^2) \cong \Z$.  Note that altering $J$ by a local knot does not alter the
isotopy class, as we can change the crossings of the local knot arbitrarily by isotopies in $S^1 \times S^2$.  We
may consider other knots in the same free homotopy class that are not isotopic to $J$, but in fact up to
concordance, we now show there is no difference.

\begin{theorem}[Concordance light bulb theorem]\label{thm:topsurgery}
  Suppose that a knot $K \subset S^1\times S^2$ lies in the free homotopy class of $x$.  Then $K$ is concordant to $J$.
\end{theorem}

\begin{proof}
  Let $X_K := Y \sm \nu K$ and let $X_J := Y \sm \nu J$, noting that $X_J\cong S^1 \times D^2$.
  Let $M_K:= X_K \cup X_J$, joined along the boundary tori with meridian mapping to meridian and longitude mapping to longitude.  Note that if $K$ and $J$ are concordant, then the boundary of the exterior of the concordance is $M_K$.

  \begin{claim}
    The homology $H_1(Y \sm \nu K;\Z) \cong \Z$.
  \end{claim}

The $\Z$--coefficient Mayer--Vietoris sequence for the decomposition $Y=(Y \sm \nu K) \cup_{S^1 \times S^1} \nu K$ yields:
\[H_2(Y) \to H_1(S^1 \times S^1) \to H_1(Y \sm \nu K) \oplus H_1(\nu K) \to H_1(Y) \to 0.\]
The generator of $H_2(Y) \cong \Z$ maps to the meridian of $K$ in $H_1(S^1 \times S^1)$.  The longitude maps onto $H_1(\nu K)$.  It follows that $H_1(Y\sm \nu K) \cong H_1(Y) \cong \Z$ as claimed.

  \begin{claim}
    The homology $H_1(Y \sm \nu K;\Z[\Z]) =0$.
  \end{claim}

To see this, note that we can understand $Y \sm \nu K$ as a Kirby diagram by considering a $2$--component link $L=L_1\cup L_2\subset S^3$ with linking number one, where we take $L_2$ to be unknotted and marked with a zero, and $L_1$ is defined as the knot that becomes $K$ after 0--surgery on $L_2$. Under the abelianisation $\pi_1(X_L) \to \Z$, the meridian of $L_1$ is sent to zero, while the meridian of $L_2$ is sent to a generator.  The Alexander polynomial of a 2--component link with linking number one satisfies $\Delta_L(1,t) = \Delta_{L_2}(t)$, by the Torres condition~\cite[Section~5.1]{MR2931688}.  But $L_2$ is unknotted, so $\Delta_{L_2}(t)=1$ and therefore $H_1(X_L;\Z[\Z])=0$.  Glue in the surgery solid torus to the boundary of $\nu L_2$, to obtain $Y \sm \nu K$.  This solid torus also has $H_1(S^1 \times D^2;\Z[\Z])=0$, so it follows that $H_1(Y \sm \nu K;\Z[\Z])=0$ as claimed.

\begin{claim}
  There exists a choice of framing on $M_K$ such that it is null bordant over $S^1$, in other words represents the 0 class in $\Omega_3^{fr}(S^1)$.
\end{claim}

By the Atiyah--Hirzebruch spectral sequence we have $\Omega_3^{fr}(S^1) \cong \Omega_3^{fr} \oplus \Omega_2^{fr}$. As in Davis \cite{MR2212279}, Cha--Powell \cite{MR3270170}, we can choose any framing to start, and then alter it in a neighbourhood of a point until the framing gives the zero element of $\Omega_3^{fr} \cong \Z/24\Z$. This procedure is possible because the $J$--homomorphism $\pi_3(O)\to \pi_3^S\cong \Omega^{fr}_3$ is onto.
The element in $\Omega_2^{fr} \cong \Omega_2^{spin} \cong \Z/2\Z$ represented by $M_K$ is trivial: it is not too hard to see that the surface produced by transversality is a sphere.  In $X_K$, the inverse image of this sphere is a sphere punctured by the knot $K$, potentially in several places.  Each of the punctures is bounded by a meridian of $K$, which is identified with a meridian of $J$.  But then a meridian of $J$ bounds an embedded disc in $X_J$.
 This completes the proof of the claim.

Now follow the standard procedure from Freedman--Quinn \cite{MR1201584}, Hillman \cite[\textsection 7.6]{MR2931688}, Davis \cite{MR2212279}.  The fact that $M_K$ is framed null bordant gives rise to a degree one normal map $W \to S^1 \times D^3$ which is a $\Z[\Z]$--homology equivalence on the boundary. The surgery obstruction to changing $W \to S^1 \times D^3$ into a homotopy equivalence lies in $L_4(\Z[\Z])$. We have \[L_4(\Z[\Z]) \cong L_4(\Z) \oplus L_3(\Z) \cong L_4(\Z) \cong L_0(\Z) \cong 8 \Z\]so that the surgery obstruction may be calculated as the signature of $W$. Hence we can kill the obstruction by taking $W$ connected sum with the $E_8$ manifold, with appropriate orientations, sufficiently many times.
As $\Z$ is a `good' group (in the sense of \cite{MR1201584}), we may now do surgery on our normal map to get a homotopy equivalence $W' \to S^1 \times D^3$, where $\partial W'$ is still $M_K$.

Glue in $S^1 \times D^2 \times D^1$ to part of the boundary, namely a thickening $S^1 \times S^1 \times D^1$ in $M_K=X_K\cup_{S^1\times S^1}X_J$ of the gluing torus $S^1 \times S^1$, to obtain a concordance from $K$ to $J$ in a $4$--manifold $V = W' \cup_{S^1 \times S^1 \times D^1} S^1 \times D^2 \times D^1$. Note that $V$ and $S^1 \times S^2 \times I$ have the same fundamental group and the same homology over $\Z[\Z]$, so by the Hurewicz Theorem they have the same homotopy groups.

\begin{claim}
 The $4$--manifold $V$ is homeomorphic to $S^1 \times S^2 \times I$.
\end{claim}

Cap off $V$ on the top and bottom boundaries with copies of $S^1 \times D^3$.  This creates a $4$--manifold $Z$ that has the same homotopy groups as $S^1 \times S^3$.  Then $Z$ is homeomorphic to $S^1 \times S^3$ \cite[Theorem 10.7A]{MR1201584}.
Now remove the images of the two caps $S^1 \times D^3$ in $S^1 \times S^3$.  These are isotopic to standard embeddings, so the outcome is $S^1 \times S^2 \times I$ as claimed.

This means that the concordance of $K$ to $J$ in $V$ is in fact a concordance of $K$ to $J$ in $S^1 \times S^2 \times I = Y \times I$ as required, which completes the proof of the theorem.
\end{proof}

\begin{corollary}
  If $x$ is the free homotopy class of $S^1\times\{\pt\}$ then the set $\mathcal{C}_x(S^1\times S^2)$ contains exactly one almost-concordance class.
\end{corollary}

\section{Almost-concordance and satellites}

We will recast almost-concordance as the orbit relation of a satellite action. The
construction of satellite knots in $S^3$ can be described in many equivalent
ways -- here is a generalisation for one of these satellite constructions to
knots in a general 3--manifold.
A \emph{knot framing}~$\psi$ of $K$ is an embedding~$\psi \colon S^1 \times D^2 \to Y$
such that $\psi(S^1 \times \{0\})$ is $K$.
Associated to a framing is the longitude~$\lambda_\psi = \psi(S^1 \times \{\pt\})$.

A framed knot~$(P, \psi)$ with $P \subset \interior(S^1\times D^2)$ is called a \emph{framed pattern}.
For a framed knot $(K, \psi_K)$ in $Y$ and a framed pattern $(P, \psi_P)$, the associated \emph{satellite knot} is
the framed knot $P(K) = \psi_K(P) \subset Y$ with framing $\psi_K \circ \psi_P$.
The set $\mathcal{P}$ of framed patterns with the operation~$P\cdot Q:= P(Q)$,
which is called \emph{satellite action}, is a monoid. This monoid~$\mathcal{P}$
acts on the set of framed knots~$\operatorname{FrKnots}(Y)$ via
\[\mathcal{P}\times \operatorname{FrKnots}(Y)\to \operatorname{FrKnots}(Y);\qquad (P,K)\mapsto P(K).\]
As before, we denote the set of framed knots in the homotopy class~$x\in [S^1, Y]$
by $\operatorname{FrKnots}_x(Y)$.

The \emph{winding number} of a pattern is the unique $n\in \Z$ such that the
knot represents $n$ times the positive generator of $H_1(S^1\times D^2;\Z)$.
Patterns with winding number~$1$ form a submonoid of $\mathcal{P}$.
Suppose that $P$ is such a pattern. Then $P(K)$ is always freely homotopic to $K$.
This follows from the observation that, for patterns with winding number~$1$,
$P$ is a generator of $\pi_1(S^1\times D^2)$.

Given any knot $J$ in $S^3$, form a pattern $P_J\subset\interior(S^1\times D^2)$ by removing from $S^3$ a small open tubular neighbourhood of the meridian to $J$. (We identify the exterior $S^3\sm \nu U$ of any unknot $U\subset S^3$ with $S^1\times D^2$ by mapping the meridian of $U$ to $S^1\times \{\pt\}$ and a 0--framed longitude of $U$ to $\{\pt\}\times \partial D^2$.) The pattern $P_J$ is taken to be canonically framed using the $0$--framing of $J$ in $S^3$.  A pattern $P$ obtained in this way has winding number~$1$.

\begin{proposition}\label{prop:PatternToSum}
Let $Y$ be a 3--manifold and fix $x\in [S^1, Y]$.
Then the underlying unframed knot of $P_J(K)$ is $K\# J$ and is thus independent of the framing of $K$.
Moreover, the group action
\[\begin{array}{rcl}
\mathcal{C}\times\mathcal{C}_x(Y)&\to& \mathcal{C}_x(Y)\\
(J, K) &\mapsto& P_J(K).
\end{array}\]
is well-defined and agrees with the local action of Definition~\ref{def:local}.
\end{proposition}

\begin{proof}
It is enough to show that for $J$ in $S^3$ and $K$ in $Y$, with any framing on
$K$, we have $P_J(K)= K\# J$. The \emph{disc knot} of $J$ is
$\Delta_J\colon D^1\hookrightarrow D^3$, obtained as the knotted embedding of $D^1$ in the
exterior of small open 3--ball around a point $\pt\in J\subset S^3$. As such
$J=\Delta_J\cup\Delta_U$ for $U$ the unknot in $S^3$. But now it is clear that
the knotted part of $P_J$ can be forced into a small ball
$D^3\subset S^1\times D^2$ and hence $P_J=P_U\# J$. So regardless of the framing of $K$, the
construction of $P_J(K)$ yields $K\# J$.
\end{proof}

For the convenience of the reader, we recall the following lemma.
\begin{lemma}\label{lem:hbordism}
Let $L \subset S^3$ be a $2$--component link with linking number~$1$. Then
for each boundary component~$T$ of $S^3 \sm \nu L$, and for all $k \in \Z$, the induced map
\[ H_k(T;\Z) \xrightarrow{\cong} H_k(S^3 \sm \nu L; \Z) \]
is an isomorphism. Consequently, $S^3 \sm \nu L$ is a homology bordism.
\end{lemma}
\begin{proof}
Let $L_0$, and $L_1$ be the two components.
The claim follows from Mayer--Vietoris sequences of the
decomposition~$S^3 \sm \nu L_0 = S^3 \sm \nu L \cup \nu L_1$ and
$S^3 \sm \nu L_1 = S^3 \sm \nu L \cup \nu L_0$.
\end{proof}

Let $P \subset S^1 \times D^2$ be a framed pattern with winding number~$1$.
We say the pattern is \emph{well-framed} if the longitude $\lambda_P$ is
homologous to $S^1 \times \{\text{pt}\} \subset S^1 \times \partial D^2$
in $S^1\times D^2\sm P$.
By Lemma~\ref{lem:hbordism} the manifold $S^1 \times D^2 \sm \nu P$ is a homology bordism, so there
always exists a well-framing.

\begin{lemma}\label{lem:meridian-winding-number-1}
Let $K\in \operatorname{FrKnots}(Y)$ be a framed knot in a $3$--manifold $Y$
and $P \in \mathcal{P}$ a winding number $1$ pattern which is well-framed.
Pick tubular neighborhoods $\nu K$
and $\nu P(K)$ such that $\nu P(K) \subset \nu K$. Then the following statements hold:
\begin{enumerate}[font=\normalfont]
\item $[P(K)]=[K]\in H_1(Y;\Z)$.
\item the meridian~$\mu_K$ is homologous in $\nu K \sm \nu P(K)$ to the meridian $\mu_{P(K)}$, and
\item the longitude~$\lambda_K$ is homologous in $\nu K \sm \nu P(K)$ to the longitude~$\lambda_{P(K)}$.
\end{enumerate}
\end{lemma}
\begin{proof}
By assumption $P$ is a winding number $1$ pattern, so $P(K)$ is homologous to
$K$ already in $\nu K$ and the first statement follows.

Consider the $3$--manifold~$M = \nu K \sm \nu P(K)$.
By Lemma~\ref{lem:hbordism}, we can compose the isomorphisms induced by the inclusions
\[H_1(\partial \nu K; \Z) \xrightarrow{\cong} H_1(M; \Z)  \xleftarrow{\cong}H_1(\partial \nu P(K); \Z)\]
and obtain an isomorphism $\Phi \colon H_1(\partial \nu P(K); \Z) \to H_1(\partial \nu K; \Z)$.

For the statement concerning meridians, we consider the diagram of maps induced by inclusions:
\[
\xymatrix{
H_1(\partial \nu K;\Z)\ar[r] & H_1(\nu K;\Z)\\
H_1(\partial \nu P(K);\Z) \ar[u]_\Phi \ar[r] & H_1( \nu P(K); \Z) \ar[u]_{\cong} \\}
\]
We see that $\Phi$ restricts to an isomorphism between the kernels of the horizontal maps.
Recall that the meridian up to a sign is characterised by the kernel of the respective horizontal map
and so $\Phi$ maps the meridian of $P(K)$ to the meridian of $K$.

The longitude of $K$ is homologous to the longitude of $P(K)$ as the pattern~$P$ is well-framed.
\end{proof}

\section{Twisted Reidemeister torsion and satellites}
To talk precisely about Reidemeister torsion, we establish some algebraic
conventions and notation.

Let $R$ be a ring $R$ with unit and an involution~$r\mapsto \overline{r}$.
One example to keep in mind is the the group ring $\Z[\pi]$ for a group~$\pi$ which
carries the involution $\Sigma_gn_gg\mapsto \Sigma_gn_gg^{-1}$.

Given a left $R$--module $A$, let
$A^t$ denote the right $R$--module defined by the action $a\cdot
r:=\overline{r}\cdot a$ for $a\in A$ and $r\in R$. Similarly, we may switch
right $R$--modules to left ones, and if $S$ is another ring with involution we
may switch $(S,R)$--bimodules to $(R,S)$--bimodules. For an $(R,S)$--bimodule
$B$ and a left $R$--module $A$, the abelian group $\Hom_R(A,B)$ has a
natural right $S$--module structure. Using the natural $(R,R)$--bimodule
structure on $R$ a left $R$--module $A$ determines a right $R$--module~$A^\vee := \Hom_R(A, R)$.
A chain complex of left $R$--modules $C$ determines the dual
chain complex of right $R$--modules $C^{-*}:=\Hom_R(C_*,R)$.
Here, recall that $d^{-r}=(-1)^{r+1}d_{r}^\vee:C^{-r}\to C^{-r+1}$.

A group homomorphism~$\phi \colon \pi \to R^\times$ into the units of the ring~$R$ is
called a \emph{representation}. It is called \emph{unitary} if $\phi(g^{-1}) = \ol{\phi(g)}$
for all $g \in \pi$. A unitary representation~$\phi$ induces a homomorphism
$\phi\colon \Z[\pi] \to R$ of rings with involution.
With this homomorphism, we can give $R$ the structure of a $(\Z[\pi],R)$--bimodule.

The following is straightforward to prove and is left to the reader.
\begin{lemma}\label{lem:algebra}
Let $\phi \colon \pi \to R^\times$ be a unitary representation.
Let $A$ be a left $\Z[\pi]$--module. Then the following map is well-defined and
an isomorphism of left $R$--modules.
\[\begin{array}{rcl}\Hom_{\Z[\pi]}(A,R)^t&\to& \Hom_{R}(A^t\otimes_{\phi}R,R),\\
f&\mapsto&\left(a\otimes b\mapsto \overline{f(a)}\cdot b\right).
\end{array}\]
\end{lemma}

For a CW pair $(X,Y)$ and $p:\widetilde{X}\to X$ the universal cover, write
$\widetilde{Y}=p^{-1}(Y)$. Then setting $\pi=\pi_1(X)$ and writing
$C=C_*(\widetilde{X},\widetilde{Y};\Z)$ for the chain complex of left
$\Z[\pi]$--modules, we write $C^{-*}(\widetilde{X},\widetilde{Y};\Z):=C^{-*}$
for the dual chain complex.
Given a unitary representation $\phi:\pi \to R^\times$, we define the
following left $R$--modules
\[\begin{array}{rcl}H_r(X,Y;\phi)&:=&H_r\big((C_*(\widetilde{X},\widetilde{Y};\Z)^t\otimes_{\phi}
R )^t\big),\\
H^r(X,Y;\phi)&:=&H_{-r}\big((C^{-*}(\widetilde{X},\widetilde{Y};\Z)\otimes_{\phi}
R )^t\big)\cong H_r(\Hom_{\Z[\pi]}(C_*(\widetilde{X},\widetilde{Y};\Z), R
)^t)\end{array}\]

Suppose $(X,Y)$ is an $n$--dimensional Poincar\'{e} pair. For $R=Q$ a field, we may apply twisted Poincar\'{e}--Lefschetz duality, then Lemma \ref{lem:algebra} and finally the Universal Coefficient Theorem to obtain \[\begin{array}{rcl}H_r(X;\phi)&\cong& H^{n-r}(X,Y;\phi)\\
&\cong& H_{n-r}(\Hom_{ Q }(C_*(\widetilde{X},\widetilde{Y};\Z)^t\otimes_{\phi} Q , Q ))\\
&\cong& \Hom_Q(H_{n-r}(X,Y;\phi)^t,Q).\end{array}\]
When $H_*(Y;\phi)=0$, we have $H_*(X;\phi)\cong H_*(X,Y;\phi)$, so in this case we obtain further a Poincar\'{e} duality of $Q$--vector spaces of the form:
\begin{equation}\label{eq:poincare}\tag{$\dagger$}
H_r(X;\phi)\cong (H_{n-r}(X;\phi)^t)^\vee.
\end{equation}

\subsection{Self-dual based torsion}
\label{section:definition-torsion}

We recall the algebraic setup for torsion invariants.
Suppose $C$ is a based chain complex over a field $Q$ and
$\B=\{\B_*\}$ is a basis for $H_*(C)$, i.e.\ $\B_i$ is a basis of the
$Q$--vector space $H_i(C)$.  The \emph{torsion} $\tau(C; \B) \in
Q^\times$ is defined as in \cite[Definition I.3.1]{MR1809561}.
If $H_*(C)$ is identically zero, then we will just write $\tau(C)\in Q^\times$ for the torsion.

Let $(X,Y)$ be a finite CW pair with $\pi=\pi_1(X)$ and let
$\phi
\colon \pi \to Q \sm \{0\}$ be a representation to a field~$Q$.  Let
$\mathcal{B}=\{\B_*\}$ be a basis of $H_*(X,Y;\phi)$.  The universal
cover $(\widetilde{X}, \widetilde{Y})$ has a natural cell structure, and the
chain complex $C_*(\widetilde{X}, \widetilde{Y};\Z)$ can be based over
$\Z[\pi]$ by choosing a lift of each cell of $(X,Y)$ and orienting it.
This gives rise to a basing of $C_*(\wti{X},\wti{Y};\Z)^t\otimes_\phi Q$ over~$Q$.
We can then define the
\emph{twisted torsion}
\[
\tau^\phi(X,Y;\mathcal{B})\in Q^\times
\]
to be the torsion of $C_*(\wti{X},\wti{Y};\Z)^t\otimes_\phi Q$ with respect to~$\B$.  We will
drop $\B$ from the notation if $H_*(X,Y;\phi)=0$.

\begin{remark}
The element~$\tau^\phi(X,Y;\mathcal{B})$ is well-defined up to multiplication by
an element in $\pm \phi(\pi)$, and is invariant under simple
homotopy preserving~$\B$ \cite[Section II.6.1 and Corollary II.9.2]{MR1809561}.
By Chapman's theorem~\cite{MR0391109} the
invariant $\tau^\phi(X,Y;\mathcal{B})$ only depends on the
homeomorphism type of $(X,Y)$ and the basis $\mathcal{B}$.  In particular, when $(M,N)$ is a
manifold pair, we can define $\tau^\phi(M,N;\mathcal{B})$ by picking
any finite CW structure for $(M, N)$.
\end{remark}

Now we consider a special case of this construction and explain
how to deal with the choice of basis~$\mathcal{B}$.
Let $(M, \partial M)$ be a $3$--manifold with boundary.
We will focus in a rather special kind of representation obtained as follows.
Let $F$ be a free abelian group. Furthermore, assume we have two
group homomorphisms $\rho \colon \pi_1(M) \to \{\pm 1\} \subset \Q^\times$ and
$\alpha \colon \pi_1(M) \to F$. Denote the quotient field
of $\Q[F]$ by $\Q(F)$.
One can check directly  that the homomorphism~$\rho \otimes \alpha \colon \pi_1(M) \to \Q(F)$
is a representation, which we write as $\phi$ to save notation.

\begin{definition}
A representation $\phi \colon \pi_1(M) \to \Q(F)$ obtained
by the construction above is called \emph{sign-twisted}.
\end{definition}

Suppose $\phi \colon \pi_1(Y) \to \Q(F)$ is a sign-twisted representation such that
$H_*(\partial Y; \phi) = 0$.
Since $M$ is odd-dimensional, we can pick a basis
$\B=\{\B_*\}$ for $H_*(M;\phi)$ with the following property: for
each~$r$, $\B_{r}$ is the dual basis of $\B_{n-r}$ via the
Poincar\'e duality isomorphism (\ref{eq:poincare}).
We call such a basis $\B=\{\B_*\}$ a
\emph{self-dual basis} for $H_*(M;\phi)$.

\begin{definition}
The \emph{norm subgroup} $N(F)$ is
\[ N(F):= \{ r\cdot f \cdot q\cdot \overline{q}\,|\, r\in \Q^\times, f\in F  \text{ and }q\in \Q(F)\sm \{0\}\}. \]
The \emph{self-dual based torsion} $\tau^\phi(M)$ is the following element in the quotient
\[  \tau^{\phi}(M; \mathcal{B}) \in \Q(F)^\times / N(F). \]
\end{definition}

\begin{remark}We note the following about the preceding definition.
\bn
\item
The self-dual based torsion is indeed well-defined.
Switching from one self-dual basis to another
changes the torsion by an element of the form
$\pm q\overline{q}$ with $q\in \Q(F)\sm \{0\}$~\cite[Lemma 2.3]{MR3062861}.
Furthermore, different choices of lifts of the cells change the
torsion only by $\pm\alpha(\pi_1(M))$~\cite[Section II.6.1]{MR1809561}.
\item Note that $N(F)$ is \emph{not} just the group of ``norms'' $q\cdot \overline{q}$, but also their products with all elements of $F$ and also all non-zero rational numbers. The fact that the rational numbers are also contained plays a r\^ole later on, in Proposition~\ref{prop:mutorsion}.
\en
\end{remark}

Now we provide a way to distinguish elements in the quotient
$\Q(F)^\times / N(F)$ by constructing epimorphisms to $\Z/2\Z$.
Recall that $\Z[F]$ is a unique factorisation domain.
It follows that given any irreducible
 polynomial $g$ over $\Z[F]$ we have a well-defined monoid
homomorphism
\[ \ba{rcl} \Phi'_g\colon \Z[F]^\times &\to & \N_0\\
q&\mapsto & \mbox{maximal $n$ such that $g^n$ divides $q$}.\ea\]
This extends to an epimorphism
\[ \ba{rcl} \Phi'_g\colon \Q(F)^\times &\to & \Z\\
rs^{-1}&\mapsto & \Phi_g(r)-\Phi_g(s).\ea\]
We call a polyomial $g\in\Z[F]$ \emph{symmetric}, if there exist a unit $a \in \Z[F]^\times$
such that $g = a \ol g$.
If $g$ is symmetric, then for any $q\in \Z[F]^\times$ we have $\Phi_g(\ol{q})=\Phi_{\ol{g}}(q)=\Phi_g(q)$. Thus we see that $\Phi_g$ descends to an epimorphism
\[ \ba{rcl}
\Phi_g\colon \Q(F)^\times/N(F) &\to & \Z/2\Z\\{}
[rs^{-1}]&\mapsto & \Phi_g(r)-\Phi_g(s)\,\mbox{ mod } 2.
\ea \]

\begin{definition}\label{defn:ExtractNorms}
For an irreducible and  symmetric polynomial $g \in \Z[F]$,
we call $\Phi_g\colon \Q(F)^\times/N(F) \to \Z/2\Z$ the \emph{parity
homomorphism}.
\end{definition}

\subsection{Alexander polynomial of a pattern}\label{subsec:AlexanderPolynomial}

Let $P \subset \interior(S^1\times D^2)$ be a
winding number~$1$ pattern which is well-framed.
We denote the meridian of $P \subset S^1\times D^2$ by~$s$.
By Lemma~\ref{lem:meridian-winding-number-1}, the homology class of $s$ agrees with the class
$[ \{\pt\} \times \partial D^2 ] \in H_1(S^1\times D^2 \sm \nu P; \Z)$.
As $P$ is well-framed, the homology class $t$ of the longitude of $P$ is homologous
in $S^1\times D^2\sm P$ to that of the curve~$S^1 \times \{\pt\}$.

The meridian~$s$ and the longitude~$t$ determine a preferred
isomorphism~$H_1(S^1 \times D^2\sm \nu P;\Z) \cong \Z\langle s,t\rangle$,
where $\Z\langle s,t\rangle$ denotes the free abelian group on the generators $s$ and $t$.
As usual, we consider the Alexander module~$H_1(S^1 \times D^2 \sm \nu P; \Z[s^{\pm 1}, t^{\pm 1}])$ of the
pattern. This is a module over $\Z[s^{\pm 1}, t^{\pm 1}]$, and thus we can consider its order. (We refer to \cite[I.4.2]{MR1809561} for the definition of the order of a $\Z[s^{\pm 1}, t^{\pm 1}]$--module.)
We denote the order of the Alexander module by
\[ \Delta_P(s,t) \in \Z[H_1(S^1 \times D^2\sm \nu P;\Z)] = \Z[s^{\pm 1}, t^{\pm 1}], \]
and we refer to this as the \emph{Alexander polynomial} of~$P$.

Consider the standard embedding $S^1 \times D^2 \subset S^3$. The closure of its complement is
also a solid torus with core~$c$. With this embedding, we can associate
to the pattern~$P$ the two-component link~$L(P) = P \cup c \subset S^3$.

We have a diffeomorphism between the complements of $P$ and $L(P)$
\[ (S^1\times D^2)\sm \nu P \xrightarrow{\cong} S^3\sm \nu L(P),\]
which is isotopic to the inclusion. Correspondingly, we can also relate
the Alexander polynomial of~$P$ to the Alexander polynomial
of~$L(P)$.

\begin{lemma}
Let $P$ be a winding number one pattern. Then the following holds:
\begin{enumerate}[font=\normalfont]
\item Under the identification
$(S^1\times D^2)\sm \nu P=S^3\sm \nu L(P)$, the meridian of the pattern corresponds to the meridian of the first component of $L(P)$, while the longitude of the pattern corresponds to the meridian of the second component of $L(P)$.
\item We have $\Delta_P(s,t)=\Delta_{L(P)}(s,t)$, where $\Delta_{L(P)}(s,t)$ denotes the usual two-variable Alexander polynomial of the $2$--component link $L(P)$.
\end{enumerate}
\end{lemma}

\begin{proof}
The lemma follows immediately from the fact that the Alexander polynomial of a link in $S^3$ only depends on the complement together with
the meridians.
\end{proof}

For local patterns, the Alexander polynomial has the structure
described in the next lemma.

\begin{lemma}\label{lem:alexander-polynomial-of-pj}
Let $J$ be an oriented knot in $S^3$ and let $\mu$ be a meridian of $J$.
We denote the pattern that is given by removing a small
open tubular neighbourhood of~$\mu$ from~$S^3$  by~$P_J$.
Then $L(P_J)=J\cup \mu$ and
\[ \Delta_{P_J}(s,t)\overset{.}{=}\Delta_{L(P_J)}(s,t)\overset{.}{=}\Delta_J(s),\]
where $\overset{.}{=}$ denotes equality up to units in $\Z[s^{\pm 1}, t^{\pm 1}]$.
\end{lemma}

\begin{proof}
See e.g.\ \cite[Proposition~5.1]{FK08}.
\end{proof}

After this detour on Alexander polynomials, we proceed by calculating
the self-dual based torsion of a satellite.
Fix a 3--manifold $Y^3$, a class~$x\in[S^1,Y]$ and pick a knot~$K$ representing $x$.
Also pick a framing~$\psi\colon S^1\times D^2\hookrightarrow Y$ of $K$.
Note that the complement $Y\sm \nu P(K)$ is glued from two pieces along a $2$--torus
\[Y\sm \nu P(K) = Y\sm \psi(S^1\times D^2)\cup \psi\left(S^1\times D^2\sm\nu P\right)
\cong Y\sm \nu K\cup_{T^2} \left(S^1\times D^2 \sm\nu P\right).\]
The glueing formula for torsion allows us to express the torsion of the satellite
in terms of the torsion of $K$ and the torsion of the pattern~$P$.

\begin{proposition}\label{prop:glueing}
Let $Y$ be a closed oriented 3--manifold.
Let $P$ be a pattern with winding number~$1$ and let
\[\phi\colon \pi_1(Y\sm P(K))\to \Q(F)\]
be a sign-twisted representation. Suppose that $\phi\not\equiv 1$ and that it is non-trivial when restricted to $\partial \nu P(K)$.
Then
\[\tau^{\phi}(Y\sm\nu  P(K))= \tau^{\phi}(Y\sm\nu  K)
	\cdot \tau^{\phi}(S^1\times D^2\sm\nu  P)\in \Q(F)^\times/N(F).\]
\end{proposition}

\begin{proof}
Using coefficients determined by $\phi$, consider the short exact
Mayer--Vietoris sequence of chain groups of $\Q(F)$--vector spaces
\[0\to C_*(T^2)\to C_*(Y\sm \nu K)\oplus C_*(S^1\times D^2\sm \nu P)\to C_*(Y\sm \nu P(K))\to 0,\]
and choose cell bases for the chain groups. Furthermore,
choose bases $\mathcal{B'''}$, $\mathcal{B}$, $\mathcal{B'}$, $\mathcal{B''}$
for the respective homologies of these chain complexes reading left to right.

As $\phi\not\equiv 1$ and it is non-trivial when restricted to $T^2$,
we obtain that
$H_*(T^2;\phi)=0$ and $\tau^{\phi}(T^2)=1$ \cite[Lemma II.11.11]{MR1809561},
and so $\mathcal{B'''}$ is empty.
From \cite[Theorem 2.2(4)]{MR3062861}, we deduce that
\[\tau^\phi(Y\sm\nu  P(K);\mathcal{B''})=\tau^\phi(Y\sm\nu
K;\mathcal{B})\cdot\tau^{\phi}(S^1\times D^2\sm
P;\mathcal{B'})\cdot\tau(H)\in \Q(H)^\times/ N(F),\]
where $H$ is the Mayer--Vietoris sequence in homology, with coefficients determined by $\phi$. Here $H$ is
thought of as an acyclic chain complex and $\tau(H)$ is calculated using the
chain basis determined by $\mathcal{B}$, $\mathcal{B'}$, $\mathcal{B''}$.

\begin{claim}
The homology $H_*(S^1\times D^2\sm P;\phi)=0$.
\end{claim}

As $P$ has winding number 1, the inclusion $\nu P\hookrightarrow S^1\times D^2$
is a homotopy equivalence. So the Mayer--Vietoris sequence of $S^1\times
D^2=\nu P \cup_{\partial \nu P}(S^1\times D^2\sm\nu P)$ with
$(\phi)$--coefficients determines isomorphisms $H_*(\partial \nu
P;\phi)\cong H_*(S^1\times D^2\sm \nu P;\phi)$.
By assumption we have $\phi\not\equiv 1$, and considering the composition
\[ H_1(\partial \nu P; \Z) \twoheadrightarrow H_1(\nu P; \Z) \xrightarrow{\cong} H_1(S^1 \times D^2; \Z) \]
which is induced by inclusions, we also see that $\phi$ restricts
to a non-trivial representation on $\partial \nu P$. As $\partial \nu P$ is a $2$--torus, this implies that
$H_*(\partial \nu P;\phi)=0$ \cite[Lemma II.11.11]{MR1809561} and
the claim follows.

We have seen that the Mayer--Vietoris sequence $H$ is non-zero only in degree~$r=1,2$
and so consists of based isomorphisms
$A_r\colon H_r(Y\sm\psi(S^1\times D^2);\phi)\xrightarrow{\cong} H_r(Y\sm \nu P(K);\phi)$.
The proposition now follows from the next claim.

\begin{claim}
The torsion $\tau(H)\in N(F)$.
\end{claim}

By definition we have that the torsion equals $\tau(H)=\det(A_1)\cdot\det(A_2)^{-1}$.
But consider the commutative diagram, where we use the Poincar\'{e} duality isomorphisms already observed in Equation (\ref{eq:poincare}): \[\xymatrix{H_2(Y\sm\psi(S^1\times D^2);\phi)\ar[d]^-{\Hom_Q(-^t,Q)}_-{\cong}\ar[rr]^-{A_2}_-{\cong}&&H_2(Y\sm\nu P(K);\phi)\ar[d]^-{\Hom_Q(-^t,Q)}_-{\cong}\\
(H_2(Y\sm\psi(S^1\times D^2);\phi)^t)^\vee&&(H_2(Y\sm\nu P(K);\phi)^t)^\vee\ar[ll]_-{(A_2)^\vee}^-{\cong}\\
H_1(Y\sm\psi(S^1\times D^2);\phi)\ar[u]_-{\text{PD}}^-{\cong}\ar[rr]^-{A_1}_-{\cong}&&H_1(Y\sm\nu P(K);\phi)\ar[u]_-{\text{PD}}^-{\cong}
}\]
As $\mathcal{B}$ and $\mathcal{B}''$ are each self-dual bases, the Poincar\'{e} duality arrows are given by the identity matrix in this basis. The matrix for $(A_2)^\vee$ is the transpose dual matrix $\ol{{A_2}^t}$, so from the bottom square we deduce that $\ol{{A_2}^t} A_1 = \id$, whence $\det(A_2)^{-1}=\det({A_2}^t)^{-1}=\overline{\det(A_1)}$, and so $\tau(H)$ is a norm as required.  This completes the proof of the claim and therefore of the proposition.
\end{proof}

We can express the factor~$\tau^{\phi}(S^1\times D^2\sm\nu  P)$ in terms
of the Alexander polynomial of the link~$L(P)$ introduced earlier in this section.

\begin{proposition}\label{prop:tors}
Let $P$ be a pattern with winding number~$1$ and a sign-twisted
representation~$\phi\colon \pi_1(S^1\times D^2\sm \nu P)\to \Q(F)$ with
associated map~$h \colon H_1(S^1\times D^2\sm \nu P; \Z) \to \Q(F)$.
Then
\[\tau^{\phi}(S^1\times D^2\sm\nu  P)= \Delta_{L(P)}(h(s),h(t))\in \Q(F)/N(F),\]
where
$s$ is the meridian of the pattern and
$t$ is the longitude of $S^1 \times D^2$.
\end{proposition}

\begin{proof}
The space $S^1\times D^2\sm \nu P$ is homeomorphic to the exterior of the
2--component link $L$ in $S^3$, consisting of the embedded pattern $P$ in $S^3$
together with an embedded loop $\{\pt\}\times \partial D^2$ for any choice
$\pt\in S^1$.
Note that $H_1(S^1\times D^2\sm \nu P;\Z) =: H$ is free abelian.
We already showed in the proof of Proposition~\ref{prop:glueing} that $H_*(S^1\times D^2\sm P;\phi)=0$,
so in fact this torsion can be calculated using torsion results in the acyclic chain complex setting.
We write $\psi:\pi_1(S^1\times D^2\sm \nu P)\to H\subset \Q(H)^\times$ for the abelianisation map.

The torsion can be expressed in terms of generators of the order ideals~\cite[Theorem 4.7]{MR1809561}
as follows
\[\tau^\psi(S^1\times D^2\sm \nu P)=\prod_{i=0}^2\ord(H_i(S^3\sm L;\psi))^{(-1)^{i+1}}
	\in\Q(H)^\times/N(H).\]
(Strictly speaking this equality only holds if the right-hand side is non-zero, but we will see in a few lines that this is the case.)
But as $S^3\sm \nu L$ is a $3$--manifold with non-empty boundary and $\rk H>1$ we conclude  that
$\ord(H_0(S^3\sm L;\psi))=\ord(H_2(S^3\sm L;\psi))=1$ \cite[Proposition 3.2~(5)~and~3.2~(6)]{MR2777847}.
We have
\[\tau^\psi(S^1\times D^2\sm \nu P)=\ord(H_1(S^3\sm L;\psi))
=\Delta_{L(P)}(s,t)\in \Q(H)^\times.\]
But $h$ induces a map $h:\Q(H)\to \Q(F)$ and under this map
$h(\tau^\psi(S^1\times D^2\sm \nu P))=\tau^\phi(S^1\times D^2\sm \nu P)$
\cite[Proposition I.3.6]{MR1809561}.
Hence we have
\[\tau^\phi(S^1\times D^2\sm \nu P)=h(\Delta_{L(P)}(s,t))=\Delta_{L(P)}(h(s),h(t))\in \Q(F)^\times/N(F).\]
By the Torres condition~\cite[Section~5.1]{MR2931688}, $\Delta_{L(P)}(1,1)$ is equal to the linking number of $L(P)$, so in particular the right-hand side is non-zero. Now  the proof is complete.
\end{proof}

\section{Topological almost-concordance invariants}

In this section, we describe how self-dual based torsion gives rise to an almost-concordance invariant.
In the $3$--sphere $S^3$, the meridian of a knot~$K$ always defines a non-torsion class
in $H_1(S^3 \sm \nu K; \Z)$. As we will see in the following proposition, in a general $3$--manifold~$Y$, this might not be the case.

\begin{proposition}\label{Prop:CurveDuality}
Let $Y$ be a closed oriented 3--manifold.
Let $x \in [S^1, Y]$ be a free homotopy class.
Let $K$ be a knot in the homotopy class~$x$.
Suppose that $[x]$ has infinite order in $H_1(Y;\Z)$. Then the following holds:
\begin{enumerate}[font=\normalfont]
\item\label{item:curve-duality-1} The meridian $\mu$ of $K$ represents a torsion element in $H_1(Y\sm \nu K;\Z)$.
\item\label{item:curve-duality-2} If $[x]\in H_1(Y;\Z)$ equals  $pa$ for some prime $p$ and  $a\in H_1(Y;\Z)$,
then the meridian represents a non-zero element in $H_1(Y\sm \nu K;\Z_p)$.
\end{enumerate}
\end{proposition}

This proposition is surely well-known to the experts, but we include a proof for the convenience of the reader.

\begin{proof}
Let $[x]\neq 0\in H_1(Y;\Z)$ be of the form $[x]=n\cdot u$ where $u$ is a primitive element of $H_1(Y;\Z)$ of infinite order.
Let $K$ be a knot representing $x$. Let $\mu$ be its meridian and pick a longitude $\lambda$. By a slight abuse of notation, we denote the corresponding elements in the various homology groups by the same symbol.
In the following we identify the boundary torus of $Y\sm \nu K$ with the product $\mu\times \lambda$.

We first consider the Mayer--Vietoris sequence with $\Z$--coefficients
\[\dots \,\to \, H_1(\mu \times \lambda;\Z)\to H_1(Y\sm \nu K;\Z)\oplus H_1(\nu K;\Z)\to H_1(Y;\Z)\to 0.\]
Since $\lambda$ has infinite order in $H_1(Y;\Z)$, it also has infinite order
in $H_1(Y\sm \nu K;\Z)$. Also note that $\mu=0$ in $H_1(Y;\Z)$ since the
meridian bounds a disc in $Y$.

Recall that the half-live-half-die lemma~\cite[Lemma 8.15]{MR1472978} says that for any orientable 3--manifold $Z$ the kernel of $H_1(\partial Z)\to H_1(Z)$ has rank one-half the first betti number of $\partial Z$.
From this lemma it follows that
$H_1(\mu\times \lambda;\Z)\to H_1(Y\sm \nu K;\Z)$ has a kernel of
rank one. Therefore the kernel is generated by an element of the form
$a[\mu]+b[\lambda]$ with $(a,b)\ne (0,0)$. Since $\mu=0$ and $\lambda\ne 0$ in
$H_1(Y;\Z)$ we have $b=0$.
Thus we have shown that $\mu$ is torsion in $H_1(Y\sm \nu K)$.  This completes the proof of (\ref{item:curve-duality-1}).

Now, to prove (\ref{item:curve-duality-2}), suppose that $p$ is a prime number such that  $[x]=pa\in H_1(Y;\Z)$  for some $a\in H_1(Y;\Z)$. We consider the same Mayer--Vietoris sequence as above, but now with coefficients in $\Z/p\Z=:\Z_p$.  We obtain:
\[H_2(Y;\Z_p)\,\to \, H_1(\mu \times \lambda;\Z_p)\to H_1(Y\sm \nu K;\Z_p)\oplus H_1(\nu K;\Z_p)\to H_1(Y;\Z_p)\to 0.\]
Since $H_1(\nu K;\Z_p)=H_1(\lambda;\Z_p)$, this sequence simplifies to
\[H_2(Y;\Z_p)\,\to \, \Z_p\langle\mu\rangle\to H_1(Y\sm \nu K;\Z_p)\to H_1(Y;\Z_p)\to 0,\]
where we recall that $\Z_p\langle\mu\rangle$ denotes the free $\Z_p$--module generated by~$\mu$.
By our hypothesis, $\lambda=0\in H_1(Y;\Z_p)$. By the exactness of the sequence, there exists a $k\in \Z_p$ such that $\lambda' := k\cdot \mu+\lambda$ is zero in $H_1(Y\sm \nu K;\Z_p)$. Note that $\mu$ and $\lambda'$ also form a basis for $H_1(\mu\times \lambda;\Z_p)$. Since $\lambda'=0$ in $H_1(Y\sm \nu K;\Z_p)$, it follows from the aforementioned half-live-half-die lemma that $\Z_p\langle\mu\rangle\to H_1(Y\sm \nu K;\Z_p)$ is injective,
hence $\mu$  is a non-zero element in $H_1(Y\sm \nu K;\Z_p)$.
\end{proof}

For an abelian group~$H$, let $FH$ denote the maximal free abelian
quotient of $H$.  We will always view $FH$  as a multiplicative group.
We consider the following knot invariant.

\begin{definition}
Let $K$ be an oriented knot in a 3--manifold $Y$.
\begin{enumerate}
\item A homomorphism $\rho\colon H_1(Y\sm \nu K;\Z/2\Z)\to \{\pm 1\} $ which is non-trivial on
the meridian is called a \emph{meridional character}. Denote the set of meridional characters
by $\mathfrak{C}(K)$.
\item Abbreviate $F:= FH_1(Y;\Z)$.
For a knot~$K$ consider the representation
\[\alpha \colon H_1(Y\sm K;\Z)\xrightarrow{i} H_1(Y;\Z)\to F,\]
which is induced by the inclusion $i$.
Let $\rho\colon H_1(Y\sm \nu K;\Z/2\Z)\to \{\pm 1\} $ be a meridional character.
Define the \emph{self-dual torsion} of $K$ to be
\[ \tau_\rho (K) := \tau^{\alpha\otimes \rho}(Y\sm \nu K) \in \Q(F)/N(F).\]
\end{enumerate}
\end{definition}

\begin{proposition}\label{prop:concordance2}
Let $Y$ be a closed oriented 3--manifold.
Let $K_0, K_1$ be two concordant knots in $Y$.
Then there exists an isomorphism $\phi\colon  H_1(Y\sm \nu K_1;\Z/2\Z) \xrightarrow{\cong} H_1(Y\sm \nu K_0;\Z/2\Z)$
which sends the meridian of $K_1$ to the meridian of $K_0$ such that for any homomorphism
$\rho\colon H_1(Y\sm \nu K_0;\Z/2\Z)\to \{\pm 1\} $ the equality below holds:
\[\tau_\rho(K_0) = \tau_{\rho \circ \phi}(K_1)\in \Q(F)^\times/N(F).\]
\end{proposition}	

\begin{proof}
Let $K_j \subset Y$, $j=0,1$, be two concordant knots and let $A\subset Y \times I$ be an annulus witnessing this.
For any concordance $A$ between $K_0$ and $K_1$, a Mayer--Vietoris argument shows that the inclusion induced map
$H_1(Y\sm \nu K_j;\Z)\xrightarrow{\cong} H_1(Y\times[0,1] \sm \nu A;\Z)$
is an isomorphism.

Denote $Y\times[0,1] \sm \nu A$ by $W$. We obtain the following commutative diagram
\[
\xymatrix{H_1(Y\sm K_0;\Z/2\Z)\ar[d]_\cong &\ar[l] H_1(Y\sm K_0;\Z)\ar[dr]^{\alpha_0} \ar[d]_\cong\\
H_1(W;\Z/2\Z)&\ar[l]H_1(W;\Z)\ar[r]&F\\
 \ar@/^6pc/[uu]^{\varphi}H_1(Y\sm K_1;\Z/2\Z)\ar[u]^\cong &\ar[l]H_1(Y\sm K_1;\Z)\ar[ur]_{\alpha_1} \ar[u]^\cong}
\]
where $\varphi$ is defined by composition of the given two vertical isomorphisms.
Note that the two inclusion maps send the meridian of the knot to the meridian of the annulus, in particular $\phi$  sends the meridian of $K_1$ to the meridian of $K_0$.

The representation $\alpha \otimes \rho\colon \pi_1( Y \sm \nu K_0) \to \Q(F)$
can be extended to $W$ as follows: define $\alpha_W \colon H_1(W;\Z) \to FH_1(Y\times I;\Z)=F$ induced
by filling in the annulus~$A$. With the
isomorphism~$H_1(Y\sm \nu K_0;\Z/2\Z)\xrightarrow{\cong} H_1(W;\Z/2\Z)$, we extend
$\rho$ over $W$ to $\rho_W \colon H_1(W;\Z/2\Z) \to \{\pm 1\}$. By the diagram, it
restricts to $\rho\circ \phi$ on $Y\sm \nu K_1$.

The boundary of $W$ is $\partial W = Y\sm \nu K_0 \cup_{T^2} Y\sm \nu K_1$.
Note that both inclusions $Y\sm \nu K_i \subset W$ induce homology equivalences. As our representation $\rho_W$ is to a 2--group, we may use \cite[Lemma 3.3]{MR3062861} to conclude that the (equivariant) intersection form of $W$, with $(\alpha_W\otimes\rho_W)$--coefficients vanishes (indeed, the underlying module is trivial). This claim allows us to use \cite[Theorem 2.4]{MR3062861} to conclude that the torsion~$\tau^{\alpha_W\otimes \rho_W}(\partial W) \in N(F)$ is contained in the norm subgroup.

Use the multiplicativity of  Reidemeister torsion corresponding to decompositions of spaces and
the fact that $\tau_\rho(T^2)$ vanishes as the representation $\rho_W$
is non-trivial on the meridian of $A$~\cite[Lemma 11.11]{MR1809561},
to obtain the equation:
\[ \tau^{\alpha_W\otimes \rho_W}(\partial W) = \tau_\rho(Y\sm \nu K_0) \cdot \tau_{\rho\circ \phi}(Y\sm \nu K_1).\]

Multiply the equation above
with $\tau_{\rho\circ \phi}(Y\sm \nu K_1)^{-1}\ol{\tau_{\rho\circ \phi}(Y\sm \nu K_1)^{-1}}$.
We then use the fact that $\tau_{\rho\circ \phi}(Y\sm \nu K_1) = \ol{\tau_{\rho\circ \phi}(Y\sm \nu K_1)}$ ---
\cite[Corollary II.14.2]{MR1809561} and \cite[Lemma~2.5]{MR3062861} ---
to obtain the desired equality
\[ \tau_\rho(Y\sm \nu K_0) \,\,=\,\, \tau_{\rho\circ \phi}(Y\sm \nu K_1) \in \Q(F)/ N(F).\]
\end{proof}

The self-dual torsion also behaves well with respect to local knotting.
Recall that to a knot~$J$ in $S^3$, we can associate a well-framed winding
number one pattern by removing a meridian such that $P_J(K) = K \# J$, see
Proposition~\ref{prop:PatternToSum}.

\begin{proposition}\label{prop:mutorsion}
Let $Y$ be a closed oriented 3--manifold. Denote $F:= FH_1(Y;\Z)$.
Let $K$ be a knot in $Y$ such that $[K]\in H_1(Y;\Z)$ is non-torsion
and $J$ a knot in $S^3$ with corresponding pattern $P$. Pick
neighbourhoods~$\nu P(K) \subset \nu K$.
Let
\[ \rho\colon H_1(Y\sm \nu (K\# J))=H_1(Y\sm \nu P(K);\Z/2\Z)\to \{\pm 1\}\] be a homomorphism
which is non-trivial on the meridian. Denote
\[ \rho' \colon H_1(Y \sm \nu K;\Z/2\Z) \xrightarrow{\cong} H_1(Y\sm \nu P(K);\Z/2\Z) \xrightarrow{\rho} \Z/2\Z, \]
which is induced by the inclusion.
Then
\[\tau_\rho(K \# J) \,\,=\,\, \tau_{\rho'}(K) \in \Q(F)^\times/N(F).\]
\end{proposition}

\begin{proof}
The inclusion induced map $H_1(Y \sm \nu K;\Z/2\Z) \xrightarrow{\cong} H_1(Y\sm \nu P(K);\Z/2\Z)$
is an isomorphism as $\nu K \sm \nu P(K)$ is a homology bordism, as shown in Lemma~\ref{lem:hbordism}.

Let $\alpha \colon H_1(Y\sm \nu P(K);\Z) \to F$ be the homomorphism induced by the inclusion.
Let $s$ be the meridian of~$P(K)$. We compute
\begin{align*}
\tau_{\rho}(Y\sm\nu  P(K)) &= \tau_{\rho'}(Y\sm\nu  K) \cdot \tau_{\rho}(\nu K \sm\nu  P(K))\\
& = \tau^{\alpha\otimes \rho'}(Y\sm \nu K)\cdot \Delta_J( (\alpha\otimes \rho)(s) ) \in \Q(F)^\times/N(F),
\end{align*}
where the first equality follows from Proposition~\ref{prop:glueing} and
the second from Lemma~\ref{lem:alexander-polynomial-of-pj}.

We make the following observations:
\begin{enumerate}
\item The character $\rho$ takes values in $\pm 1 \in \Q$ and the map $\alpha$ takes values in $F$, thus
$\alpha\otimes (\rho'\circ \phi)(s)$ is of the form $\pm f$ with $f\in F$.
\item By Lemma~\ref{lem:meridian-winding-number-1}, the homology class of the meridian~$s$
in $H_1(Y\sm \nu P(K);\Z)$ equals to the one of the meridian of $K$.
We had assumed that $[K]$ has infinite order in homology. It follows from
Proposition~\ref{Prop:CurveDuality} that $\alpha(s)$ is trivial in $F$.
In particular, we see that $(\alpha\otimes \rho)(s) = \pm 1$.
It is well-known that for the Alexander polynomial of a knot $J$,
the integer~$\Delta_J(\pm 1)$ is odd, in particular non-zero.
\end{enumerate}
Summarising, $\Delta_J(\alpha\otimes \rho)(s) \in \Q^\times \subset N(F)$.
This concludes the proof of the proposition.
\end{proof}

\begin{corollary}\label{cor:TorsionInvariant}
Let $K$ be a knot in a closed, oriented 3--manifold $Y$ such that $[K]\in H_1(Y;\Z)$ is non-torsion. Let $\mathfrak{C}(K)$ be the set of
homomorphisms~$H_1(Y\sm \nu K;\Z/2\Z)\to \{\pm 1\}$ which are non-trivial on the meridian.
Then the set of self-dual torsions
\[ I_K = \{ \tau_\rho(K) \mid \rho \in \mathfrak{C}(K)\} \subset \Q(F)/N(F)\]
is an almost-concordance invariant.
\end{corollary}


\section{Changing the almost-concordance class using satellites}\label{sec:changing}

In this section we apply our almost-concordance invariants; using a satellite
construction to modify certain knots within their free homotopy classes, we will
produce infinite families of examples that serve to confirm Conjecture~\ref{conj:main} in many cases.
For our particular satellite constructions, we will make use of the
patterns $P_n$, $n\geq 1$, with winding number one, shown in Figure~\ref{fig:mazur}.
\begin{figure}[h]
\labellist
\pinlabel $n$ [l] at 265 234
\endlabellist
\includegraphics[width=0.6\textwidth]{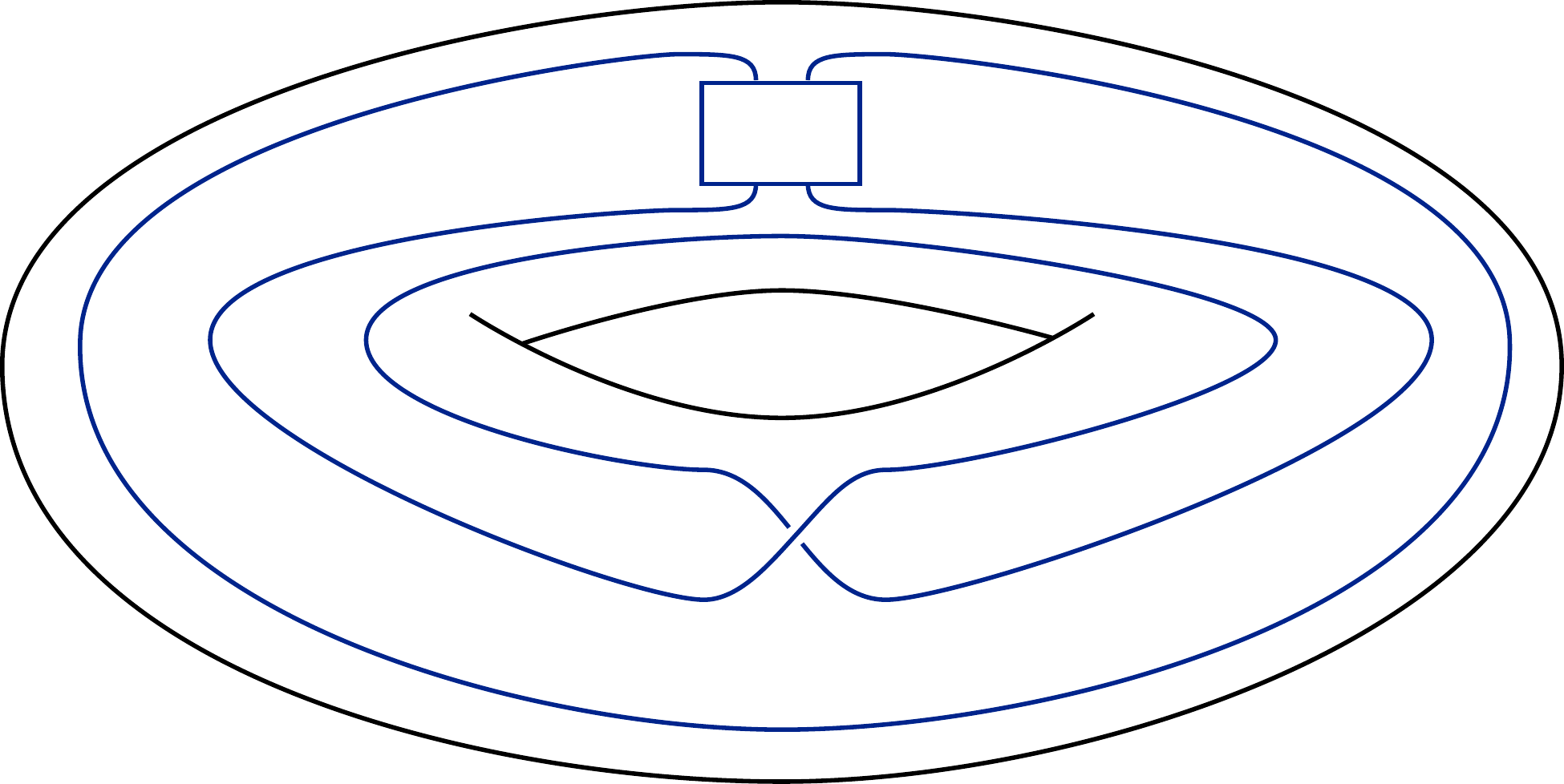}
\caption{The Mazur patterns $P_n$, where the $n$--box denotes $n$ full right-handed twists.}
\label{fig:mazur}
\end{figure}

Following Cooper \cite{Cooper82} and Cimasoni--Florens~\cite{MR2357695},
we compute the multivariable Alexander polynomials of
$P_n$ using C--complexes.
Recall that a C--complex for an $m$--component link $L$ consists of a choice of
Seifert surface~$F_j$ for each link component~$L_j$, where the surfaces are
allowed to intersect one another, but only in clasp singularities as depicted in
Figure~\ref{fig:clasps}.
\begin{figure}[h]
\includegraphics[width=0.3\textwidth]{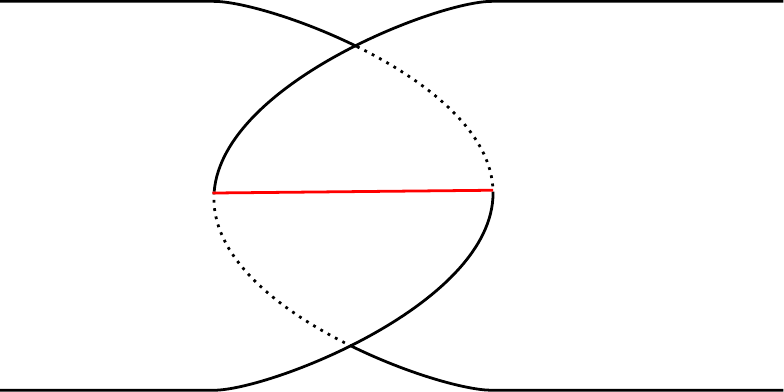}
\caption{Clasp singularities between two surfaces.}
\label{fig:clasps}
\end{figure}
In particular,
there are no triple intersections. The Seifert
form of the C--complex is defined to have underlying $\Z$--module $H_1(F;\Z)$,
where $F$ is the union of the surfaces~$F_j$. To describe the Seifert pairing
on this module, we will first pick a normal direction for each component $F_j$.
The ways an embedded curve in $F$ can be pushed into the complement $S^3\sm F$
are then encoded as a choice of function
$\eps\colon \{1,\dots, N\} \to \{0, 1\}$, where $\eps(j)=0$ and $\eps(j)=1$ stand for negative and positive
push-offs respectively from the component $F_j$. Denote the resulting push-off for
an embedded curve $x\subset F$, by $i_\eps x \subset S^3 \sm F$. Now
define a pairing on $H_1(F;\Z)$ via the formula
\[ \beta(x,y) := \sum_{\eps} (-1)^{|\eps|} \lk(i_\eps x, y) X^\eps \in \Z[X_1,\dots, X_N],\]
where $|\eps| := \sum_j \eps(j)$ and $X^\eps := \prod_j X^{\eps(j)}_j$.

The complement of a standard solid torus in $S^3$ is a
neighbourhood of an unknot $c$, so when considering the 2--component link $L(P)$
in $S^3$ arising from a winding number~$1$ pattern $P\subset S^1\times D^2$ and
this unknot $c$, we always write $s$ for the meridian of $P$
and $t$ for the meridian of $c$, see Section~\ref{subsec:AlexanderPolynomial}.

\begin{proposition}\label{prop:calculate}

For $n\geq 1$, the Mazur pattern $P_n$ has multivariable Alexander polynomial
\[ \Delta_{L(P_n)}(s,t)=n(s^2t+st^2-s^2-t^2+s+t)-(2n-1)st \in \Z[s^{\pm 1},t^{\pm 1}].\]
\end{proposition}

\begin{proof}
Consider the C--complex with the generators of $H_1(F;\Z)$ as sketched in Figure~\ref{fig:StDiagramMazur}.
\begin{figure}[h]
\labellist
\pinlabel $s$ [l] at 220 75
\pinlabel $t$ [l] at 460 70
\pinlabel $e_0$ [l] at 20 395
\pinlabel $e_1$ [l] at 530 210
\pinlabel $n$ [l] at 288 395
\endlabellist
\includegraphics[width=0.7\textwidth]{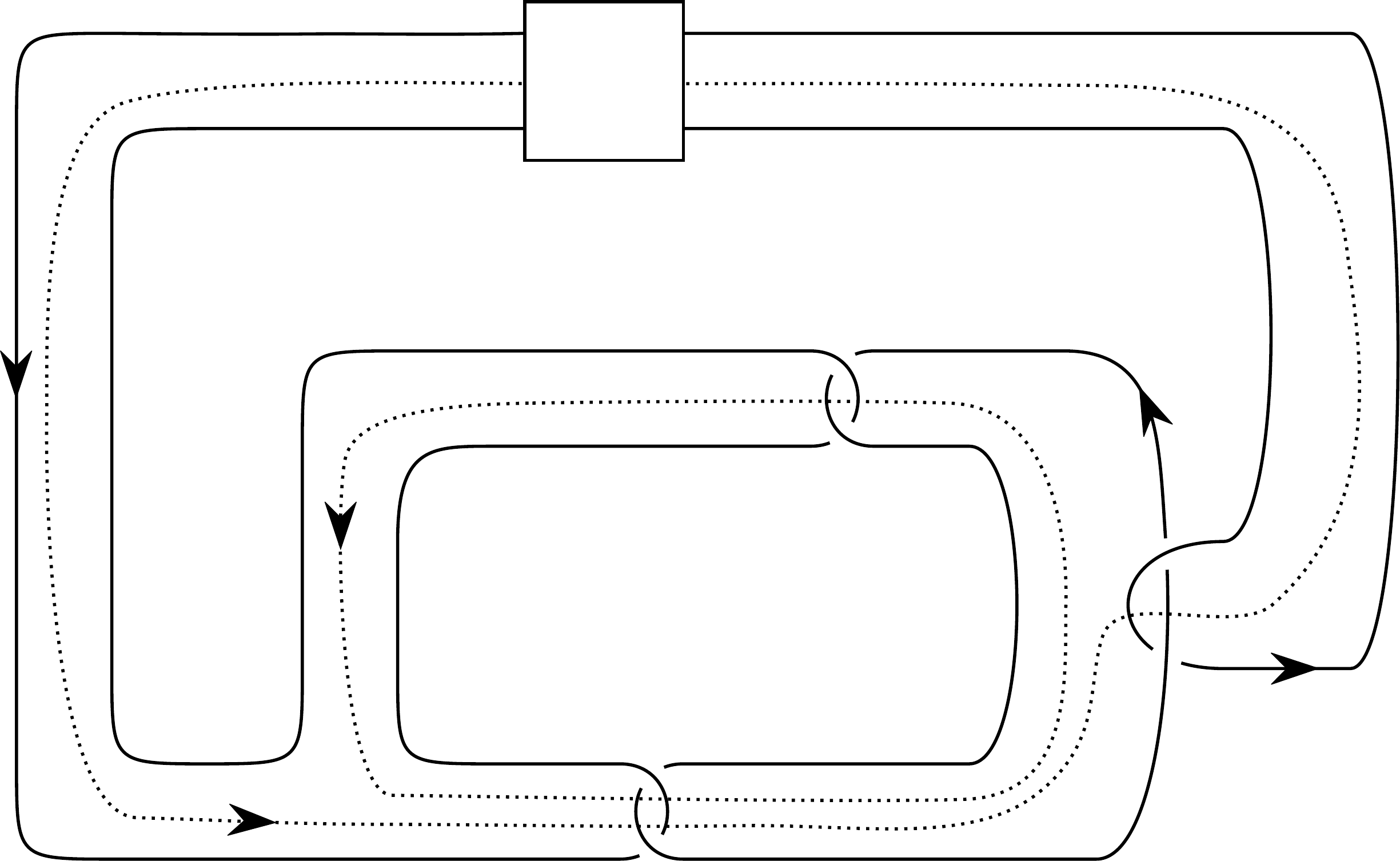}
              \caption{Diagram for $C$--complex of the Mazur pattern~$P_n$.}
	       \label{fig:StDiagramMazur}
\end{figure}
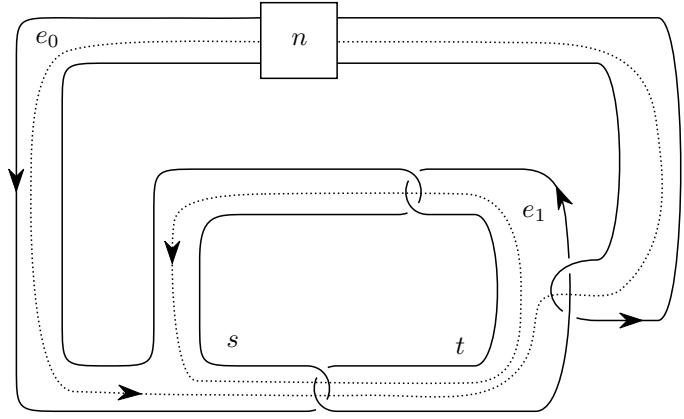

For the self-intersection, we read off the contributions and obtain:
\begin{align*}
\beta(e_0, e_0)
&= n(1-t)(1-s) + \big({-\tmfrac{1}{2}}\big)(s+t) + \tmfrac{1}{2}(s+t)\\
&= n(1-t)(1-s),\\
\intertext{and}
\beta(e_1, e_1)
&= \big({-\tmfrac{1}{2}}\big)(s+t) + \big({-\tmfrac{1}{2}}\big)(s+t)\\
&= -(s+t).
\end{align*}
The value of $\beta(e_0, e_1)$ can be computed as follows
\[ \beta(e_0, e_1)
= \tmfrac{1}{2}(1-s) - \tmfrac{1}{2}(s+t) + \big({-\tmfrac{1}{2}}\big)(1-t) = -s. \]
Consequently, we also obtain $\beta(e_1, e_0) = -t$.
We can compute the multivariable Alexander polynomial~$\Delta_{L(P_n)}(s,t)$
in the ring~$\Z[s^{\pm 1}, t^{\pm 1}, (1-s)^{-1}, (1-t)^{-1}]$
as the determinant~\cite[Corollary 3.6]{MR2357695}
\[ \det \begin{pmatrix} n(1-t)(1-s) & -s \\ -t & -(s+t) \end{pmatrix} = -n ( s^2 t + st^2 - s^2 -t^2 + s +t) +
	(2n-1) st.\]
We must now show that the equality holds moreover in $\Z[s^{\pm 1}, t^{\pm 1}]$.
The link $L(P_n) \subset S^3$ has two components and we denote the component
corresponding to $P_n$ with $\ol P$.
From Lemma~\ref{lem:alexander-polynomial-of-pj}, we deduce that
$ | \Delta_{L(P_n)}(1,1) | = | \Delta_{\ol P}(1)| = 1$.
As also \[ | -n ( 1 + 1 - 1 - 1 + 1 +1)+ (2n-1)| = 1,\] we obtain that in $\Z[s^{\pm 1}, t^{\pm 1}]$
we have
\[ \Delta_{L(P_n)}(s,t)= -n(s^2t+st^2-s^2-t^2+s+t)+(2n-1)st.\]
\end{proof}

For $F$ a finitely generated torsion-free abelian group, we call two non-zero $g, h \in \Z[F]$
\emph{associates} if $g = f \cdot h$ for a unit $f \in \Z[F]^\times$.

\begin{lemma}\label{lem:alexander-polynomials-of-patterns}
Let $F$ be a finitely generated torsion-free abelian group and let $u$ be a primitive element in $F$.
Then for any prime $p\neq 2$ the polynomial $\Delta_{P_p}(-1,u^2)$ is irreducible in $\Z[F]$. Furthermore for two different primes $p$ and $q$, the elements
$\Delta_{P_p}(-1,u^2)$ and $\Delta_{P_q}(-1,u^2)$ in $\Z[F]$ are non-associates.
\end{lemma}

\begin{proof}
We can extend~$\{u\}$ to a basis $\{u,v_1,\dots,v_n\}$ for the torsion-free abelian group $F$.
As before we use multiplicative notation for $F$.
By Proposition~\ref{prop:calculate} we have
\[ \Delta_{P_p}(s,t)\,\,=\,\, p ( s^2 t + st^2 - s^2 -t^2 + s +t) -
	(2p-1) st.\]
Thus we have
\[ \Delta_{P_p}(-1,u^2)\,\,=\,\, p ( u^2 -u^4 -1 -u^4-1 +u^2) +
	(2p-1) u^2=-2p(u^4+1)+(4p-1)u^2.\]
Since units~$\Z[F]^\times$ are of the form $\pm f$ for $f \in F$, we obtain that
for two different primes $p$ and $q$ the elements
$\Delta_{P_p}(-1,u^2)$ and $\Delta_{P_q}(-1,u^2)$ in $\Z[F]$ are non-associates.	

Now we  show that $g(u):=-2p(u^4+1)+(4p-1)u^2$ is irreducible in $\Z[u^{\pm 1}]$.
The reader may verify that the polynomial $g(u)$ has no real roots. In particular,
it cannot have an irreducible factor over $\Z[u^{\pm 1}]$ of degree 1.
Suppose, for a contradiction, that $g(u)=r_1(u)\cdot r_2(u)$ for polynomials $r_i(u)=a_iu^2+b_iu+c_i$, $i=1,2$.

\begin{claim}
  We have $a_i=\varepsilon c_i$, for $i=1,2$ and some $\varepsilon\in\{\pm 1\}$.
\end{claim}

To see the claim, multiply out the product $r_1(u)\cdot r_2(u)$ and compare coefficients with $g(u)$.
We obtain $a_1 a_2 = -2p = c_1 c_2$ from the coefficients of $u^4$ and $1$. Observe that this implies that none of the $a_i$ nor the $c_i$ can be zero.
We also obtain $b_1a_2 + a_1b_2 =0 = b_1c_2 + b_2 c_1$, from the coefficients of $u$ and $u^3$, so that either $b_1=0$ or $a_2/a_1 = -b_2/b_1 = c_2/c_1$.  If $b_1=0$, then since $a_1 \neq 0$, we also have $b_2=0$.  But if $b_1=b_2=0$, then from comparing coefficients of $u^2$ we have $a_2c_1 + a_1c_2 = 4p-1$, and since $p$ is prime one can check the possibilities for the $a_i$ and the $c_i$ that satisfy all of $a_2c_1 + a_1c_2 = 4p-1$, $a_1 a_2 = -2p$ and $c_1c_2 = -2p$, to see that this set of equations has no solutions.  Thus $b_1 \neq 0$, and so we have the equations $a_1a_2=c_1c_2$ and $a_2/a_1 = c_2/c_1$.  Rearrange the first equation to yield $a_1/c_1=c_2/a_2$.  Rearrange the second equation and substitute, to give $a_2 = c_2 a_1/c_1 = c_2 c_2/a_2$, so $c_2^2 = a_2^2$.  Thus $c_2 = \varepsilon a_2$ for some $\varepsilon\in\{\pm 1\}$. It follows that similarly we have $c_1 = \varepsilon a_1$. This completes the proof of the claim.

By the claim, we may rearrange $r_i=a_i(u^2+\varepsilon)+b_i u$.
Considering that $a_1a_2=-2p$, we will assume (after relabelling) that $p$ divides $a_1$, and as $p\neq 2$, this means $p$ does not divide $a_2$.
The coefficient of $u^3$ in $g$ is $a_1b_2+a_2b_1=0$ and hence $p$ divides $b_1$.
This is now a contradiction as we have $p$ dividing $r_1$ and hence $g$, but on the other hand the residue of $g$ in $\Z_p[u]$ is $u^2$.
We conclude from this that $g(u)$ is irreducible in $\Z[u]$. Hence also in $\Z[u^{\pm 1}]$.

It is elementary to show that this implies that $g(u)$ is also irreducible in $\Z[F]=\Z[u^{\pm 1}][v_1^{\pm 1},\dots,v_n^{\pm 1}]$. To wit, given $R$ a unique factorisation domain, and some irreducible $q\in R$, to show $q$ is also irreducible over $R[x^{\pm1}]$ suppose that $q=r(x)\cdot s(x)$. Then by looking at the degrees we see that $r(x)$ and $s(x)$ are of the form $ax^n$, $bx^{-n}$, respectively, for some $a,b\in R$. But then $q=a\cdot b\in R$ and so one of $a,b$ is a unit in $R$, therefore one of $r,s$ is a unit in $R[x^{\pm 1}]$.
\end{proof}

Recall the parity homomorphism~$\Phi_g\colon \Q(F)^\times/ N(F) \to \Z/2\Z$ for an irreducible, symmetric and
non-constant polynomial~$g \in\Z[F]$ from Definition~\ref{defn:ExtractNorms}.

\begin{theorem}\label{thm:NonTorsionCase}
Fix $(Y,x)$ and let $K$ be a knot representing $x$. Suppose that $[x]=2u\in H_1(Y;\Z)$, for some primitive homology class $u$  of infinite order.
Then there exists an infinite set $I\subset \N$ such that for any $i\ne j$ in $I$ the knots $P_i(K)$ and $P_j(K)$ belong to  distinct almost-concordance classes within the set $\mathcal{C}_x(Y)$.
\end{theorem}

\begin{proof}
Fix a framing $\psi\colon S^1\times D^2\hookrightarrow Y$ of $K$, and define the knots $P_i(K)$, $i\in \N$.
Write $F:= FH_1(Y;\Z)$, and let $\alpha\colon H_1(Y\sm \nu K;\Z)\to F$ be the inclusion induced map.
Define
\[ G\,\,:=\,\,\prod\limits_{\rho \in \mathfrak{C}(K)} \tau^{\alpha\otimes \rho}(Y\sm \nu K)\,\in \,\Z[F],\]
where we take the product over all meridional representations $H_1(Y\sm \nu K;\Z/2\Z)\to \{\pm 1\}$
of $K$.
Note that there are only finitely many such representations and hence this is a finite product.
Define
\[ I\,:=\, \{\mbox{odd primes $p$}\,|\, \Delta_{P_p}(-1,u^2)\in \Z[F]\mbox{ does not divide $G\in \Z[F]$}\}.\]
Note that $I$ is an infinite set since the $\Delta_{P_p}(-1,u^2)$ in $\Z[F]$ are pairwise non-associates. Given $i\in I$, we write $g_i:=\Delta_{P_i}(-1,u^2)$.

Recall that the set $I_n := \{ \tau_\rho(P_n(K)) \,|\, \rho \in \mathfrak{C}(K)\}$
is an almost-concordance invariant, see Corollary~\ref{cor:TorsionInvariant}.
As a result the theorem immediately follows from the next claim.

\begin{claim}
For $n,m \in I$ the sets $I_n$ and $I_m$ only intersect if $n =m$.
\end{claim}

To prove the claim, consider an element~$\tau_\rho(P_n(K)) \in I_n$.
We compute $\tau_{\rho}( P_n(K))$ in terms of $\tau_{\rho}( P_n(K))$ and
the Alexander polynomial~$\Delta_{P_n}$.
By Lemma~\ref{lem:meridian-winding-number-1} and since $[P_n(K)]=[K]=u^2\in F$,
we can use Mayer--Vietoris for torsion to compute
the following equalities in $\Q(F)^\times/N(F)$:
\begin{align*}
\tau_{\rho}(P_n(K)) &=
\tau_{\rho}(K)\cdot \Delta_{P_n}(-1,u^2)
\end{align*}

Note that the parity homomorphism $\Phi_{g_m} \colon \Q(F)^\times/ N(F) \to \Z/2\Z$,
see Definition~\ref{defn:ExtractNorms},
vanishes on $\tau_\rho(K)$.
Consequently, we obtain
\begin{align*}
\Phi_{g_m}(\tau_\rho(P_n(K))) &= \Phi_{g_m}( \tau_\rho(K) ) + \Phi_{g_m}(\Delta_{P_n}(-1,u^2))\\
&= \Phi_{g_m}(\Delta_{P_n}(-1,u^2))\\
&=\begin{cases}
1 & n=m\\
0 & n \neq m
\end{cases}
\end{align*}
where in the last line we used Lemma~\ref{lem:alexander-polynomials-of-patterns}.
This shows the claim, which completes the proof of the theorem.
\end{proof}

\section{The null-homotopic class for spherical space forms}\label{null-homotopic-case}

In Section~\ref{section:torsion-case} we will investigate linking numbers in covering spaces, to give obstructions to almost-concordance.  Before embarking on the general version, we give a more easily digestible version for a special case, as a warm up.  The next result also follows fairly easily from Schneiderman's concordance invariant \cite{MR2012959}.

Recall that a spherical space form is a compact $3$-manifold $Y$ with universal cover $S^3$.  For example, lens spaces $L(r,s)$ are spherical space forms.

\begin{proposition}\label{prop:spherical-space-forms}
  Let $Y$ be a spherical space form and let $x$ be the null-homotopic free homotopy class.  Then if $Y\neq S^3$, the set $\mathcal{C}_x(Y)$ contains infinitely many almost-concordance classes.
\end{proposition}

\begin{proof}
  Choose a non-trivial element $g \in \pi_1(Y)$, and fix an embedded circle representing $g$.  In a solid torus neighbourhood of $g$, insert the knot $K_n$ shown in Figure \ref{fig:coveringlink1}. Note that $K_n$ represents $x$ in $Y$.

\begin{figure}[h]
\labellist
\pinlabel $n$ [l] at 271 235
\endlabellist
\includegraphics[width=0.6\textwidth]{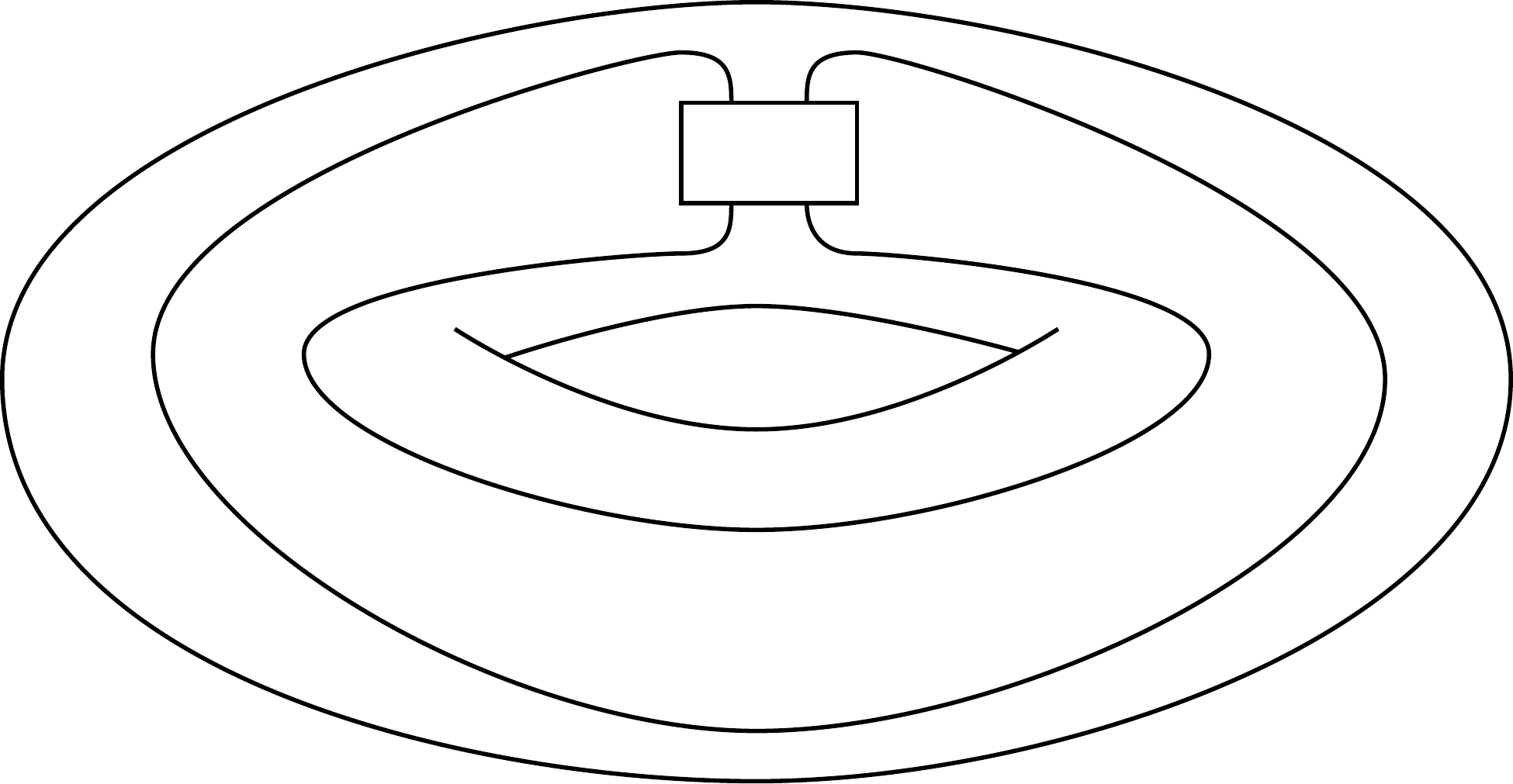}
\caption{The knot~$K_n$ where the $n$--box represents $n$ full twists.}
\label{fig:coveringlink1}
\end{figure}

Let $p \colon S^3 \to Y$ be the universal cover and consider the link given by the preimage $L_n:= p^{-1}(K_n) \subset S^3$. Linking number between components of $L_n$ is not affected by local knotting of $K_n$ in $Y$. Also, a concordance of $K_n$ in $Y \times I$ lifts to a concordance of $L_n$ in $S^3 \times I$, and linking number is a concordance invariant for links in $S^3$. Hence to show the $K_n$ are pairwise non almost-concordant, it suffices to show that for $m\neq n$, the linking numbers between the components of $L_n$ are different from those of $L_m$.

To see this, observe that the knot $K_n$ bounds an immersed disc inside the solid torus, so is null-homotopic.  Moreover we can lift this disc to see that each component of the covering link $L_n$ bounds an embedded disc in $S^3$.  The link $L_n$ has $|\pi_1(Y)|$ components. If the order of $g$ is not equal to two, then different components link with linking numbers either $n$ or $0$, as can be computed by lifting the aforementioned disc to $S^3$, and counting intersections of the other components of the lift with the lifted disc.  Specifically, $\lk(\wti{K}_n,g^{\pm 1}\cdot \wti{K}_n) = n$, and $\lk(\wti{K}_n,h \cdot \wti{K}_n) = 0$ for $h \neq g,g^{-1}$. If the order of $g$ is two, then the linking number between different lifts is either $2n$ or $0$; similarly to the generic case, $\lk(\wti{K}_n,g\cdot \wti{K}_n) = 2n$, and $\lk(\wti{K}_n,h \cdot \wti{K}_n) = 0$ for $h \neq g$.  In particular, the nonzero linking number is realised between at least two components of $L_n$.  It follows that the set of pairwise linking numbers of $L_n$ is different from the corresponding set for $L_m$, as required.
\end{proof}


\section{The torsion case when $Y$ is not the $3$--sphere}\label{section:torsion-case}

In this section, let $Y$ be a closed oriented $3$--manifold, with $Y \neq S^3$, and $x \in [S^1, Y]$ be \emph{torsion}, that is for any choice of basepoint and basing path, $x$ is finite order in $\pi_1(Y)$. We will now construct a family of pairwise non almost-concordant
knots in the torsion class $x$. As in Section \ref{sec:changing}, a satellite construction will be used to build the family, but now the almost-concordance invariant we use to distinguish the knots in the family will be based on the idea of covering links. Recall that, given a knot $K \subset Y$ and a finite covering space $p \colon \wti{Y} \to Y$, the associated \emph{covering link} is the inverse image $p^{-1}(K)$.

Observe that for each $m>0$, the local action from Definition \ref{def:local} extends to an obvious action of the $m$--fold product $\mathcal{C}\times\dots\times\mathcal{C}$ on the set of concordance classes of $m$--component links in a 3--manifold. We call the orbit of a link $L$ under this action the \emph{almost-concordance} class of $L$.

\begin{lemma}\label{lem:CoveringLinks}
Let $K,K'$ be two almost-concordant knots in $Y$ and let $\wti Y \to Y$ be a finite cover. Then
also the associated covering links are almost-concordant.
\end{lemma}

\begin{proof}
Combine the fact that a concordance between $K$ and $K'$ lifts
to a concordance in $\wti{Y} \times I$ between the covering links, and the fact that connected sum with local knots
lifts to connected sums with local knots.
\end{proof}

To take advantage of this observation, we study a notion of linking numbers for links in general 3--manifolds.
Let $N$ be a compact oriented $3$--manifold. Later on, we will specialise this theory to $N=\wti{Y}$.
For two knots $K,J \subset N$ representing torsion homology classes,
we define their \emph{linking number}~$\lk(K, J)$ as follows: pick
a class~$C \in H_2(N, \nu K; \Q)$ with $\partial C = [K] \in H_1(\nu K;\Q)$.
The relative intersection pairing
\[ H_2(N , \nu K; \Q) \times H_1(N \sm \nu K; \Q) \to \Q \]
allows us to define $\lk(K,J) := C \cdot [J]$.

\begin{claim}
The number~$\lk(K, J)$ is well-defined.  That is, it does not depend on the choice of $C$.
\end{claim}
Let $C,C'$ be distinct choices. Then there exists a class
a class~$S \in H_2(N; \Q)$ with $S  = C - C'\in H_2(N,\nu K;\Q)$.
By the commutativity of the diagram
\[
\xymatrix{
H_2(N, \nu K; \Q) \times H_1(N \sm \nu K; \Q) \ar[dr] &\\
H_2(N; \Q) \times H_1(N\sm \nu K; \Q) \ar[d] \ar[u] \ar[r] & \Q\\
H_2(N; \Q) \times H_1(N; \Q) \ar[ur] &
}
\]
and the fact that $[J]$ gets sent to $0 \in H_1(N; \Q)$, we deduce
that $C \cdot [J] - C' \cdot [J] = S \cdot [J] = 0$, which completes the proof of the claim.

\begin{lemma}\label{lem:linkingNumber}
Let $L$, $L'$ be two almost-concordant links in $N$, whose components are
torsion in $H_1(N; \Z)$.
Then the linking numbers between the components of $L$ agree with
the ones between the components of $L'$.
\end{lemma}

\begin{proof}Let $L= L_1 \cup \cdots \cup L_m \subset N \times \{0\}$, let $L'= L_1 \cup \cdots L'_m \subset N \times \{1\}$, and let $A=A_1 \cup \cdots \cup A_n \subset N \times I$ be a concordance between $L$ and $L'$.

Fix $i \neq j \in \{1,\dots,m\}$.
The annulus $A_i$ determines a class in $H_2(N \times I,N \times \{0,1\};\Q)$.  Pick a class~$C_j \in H_2(N \times \{0\}, L_j; \Q)$ with $\partial C_j = [L_j] \in H_1(L_j;\Q)$, and choose a class~$C'_j \in H_2(N\times \{1\}, L'_j; \Q)$ with $\partial C'_j = [L'_j] \in H_1(L'_j;\Q)$.  Lift $C_j$ and $C_j'$ to $2$--chains in $C_2(N \times \{k\};\Q)$, for $k=0,1$ respectively.
The sum of chains $D_j:= C_j + A_j + C_j'$ represents a class in $H_2(N \times I;\Q)$.  Since $A_i \cap A_j = \emptyset$, we have that the intersection between $D_j$ and $A_j$ is contained in the boundary $N \times \{0,1\}$, and so $[D_j] \cdot [A_i] = \lk(L_i,L_j)-\lk(L_i',L_j')$.

Next we argue that $A_i$ is zero in $H_2(N \times I ,N\times \partial I;\Q)$.
Consider the long exact sequence of a pair:
\[H_2(N \times \partial I;\Q) \to H_2(N \times I;\Q) \to H_2(N \times I,N\times \partial I;\Q) \to H_1(N \times \partial I;\Q).\]
The class $A_i \in H_2(N \times I ,N\times \partial I;\Q)$ lies in the image of $H_2(N \times I;\Q)$.  But since $N \times I$ is a product, the map \[H_2(N \times \partial I;\Q) \cong H_2(N;\Q) \oplus H_2(N ;\Q) \to H_2(N \times I;\Q) \cong H_2(N;\Q)\]
is surjective, and therefore the map $H_2(N \times I;\Q) \to H_2(N \times I ,N\times \partial I;\Q)$ is the zero map.  It follows that $A_i=0 \in H_2(N \times I, N \times \partial I;\Q)$ as desired.

The intersection pairing $\lambda \colon H_2(N \times I;\Q) \times H_2(N \times I ,N\times \partial I;\Q) \to \Q$ is (by definition) adjoint to the composition of the two isomorphisms \[H_2(N \times I;\Q) \xrightarrow{\cong} H^2(N \times I ,N\times \partial I;\Q) \xrightarrow{\cong} \Hom_{\Q}(H_2(N \times I ,N\times \partial I;\Q),\Q),\]
given by Poincar\'{e} duality and the Universal Coefficient Theorem.
Since such a pairing can be computed by geometric intersections, we compute $[D_j] \cdot [A_i]$ as $\lambda([D_j],[A_i]) = \lambda([D_j],0) =0$.
Therefore $\lk(L_i,L_j)-\lk(L_i',L_j')=0$ as required.
\end{proof}

\begin{theorem}\label{thm:linkingincover}
Let $Y \neq S^3$ be a closed, oriented $3$--manifold. Let $x \in [S^1, Y]$ be torsion
and denote the normal closure of~$x$ by $H := \langle \langle x \rangle \rangle \subset \pi_1(Y)$.
If the quotient $\pi_1(Y)/H$ is non-trivial,
then $\mathcal{C}_x(Y)$ contains infinitely
many distinct almost-concordance classes.
\end{theorem}

\begin{proof}
First we need the following fact, which uses general results about fundamental groups of 3--manifolds.

\begin{claim}
There exists an epimorphism~$ \phi \colon \pi_1(Y)/H \twoheadrightarrow G$ to some non-trivial finite group~$G$.
\end{claim}

By the Prime Decomposition Theorem and the Geometrisation Theorem, $Y\cong N_1\#\dots\# N_r$ for some closed, oriented 3--manifolds $N_i$, where for all~$i$, either $A_i:=\pi_1(N_i)$ is finite or $A_i$ is torsion free. See \cite[C.3 and \textsection\textsection 1.2, 1.7]{MR3444187} for details and references. As $x\in \pi_1(Y)\cong A_1\ast\dots\ast A_r$ is a torsion element, it must be conjugate to an element $a\in A_i$ for some $i$ \cite[Cor.\ I.1.1]{MR1954121}. By reordering, assume $i=1$. There is now an obvious epimorphism \[\Phi\colon A_1\ast\dots\ast A_r/\langle \langle x \rangle\rangle\twoheadrightarrow A_2\ast\dots\ast A_r.\]For the case $r\neq 1$, consider that the codomain of $\Phi$ is the fundamental group of a 3--manifold and is therefore residually finite, by the Geometrisation Theorem. Hence the required epimorphism $\phi$ exists. For the case $r=1$, we have that $A_1$ is finite and we may take $\phi$ as the identity map. This completes the proof of the claim.

Pick such an epimorphism~$ \phi \colon \pi_1(Y)/H \twoheadrightarrow G$ and compose with the quotient map to obtain an
epimorphism~$\phi \colon \pi_1(Y) \twoheadrightarrow G$, that vanishes on $H$.
Pick a knot~$\eta$ with $\phi(\eta)\neq 0$ nontrivial and another knot~$\alpha$
representing $x$, and disjoint from $\eta$. Pick a genus 2 handlebody~$B \subset Y$ whose cores are $\eta$ and $\alpha$.
Let $K_n \subset B$ be the knot described in Figure~\ref{fig:nullhomotopypattern}.
The homotopy class of $K_n$ agrees with the one of $\alpha$ and therefore $K_n$ represents the class $x$.

\begin{figure}[h]
\labellist
\pinlabel $\alpha$ [l] at 238 70
\pinlabel $\eta$ [l] at 455 33
\pinlabel $n$ [l] at 438 121
\endlabellist
\includegraphics[width=0.9\textwidth]{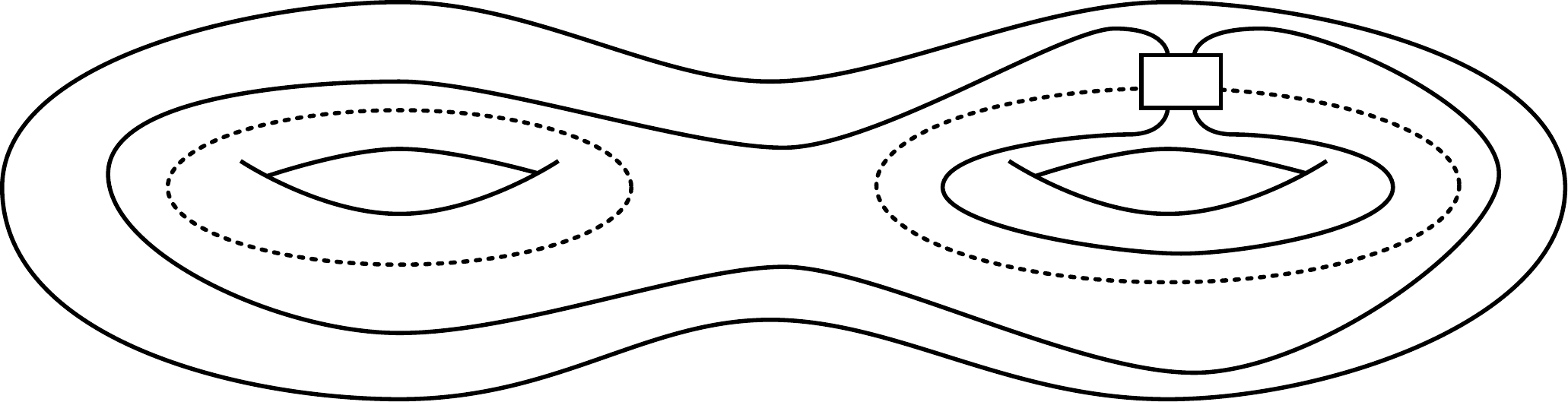}
\caption{The knot~$K_n$ where the $n$--box represents $n$ full twists.}
\label{fig:nullhomotopypattern}
\end{figure}

We show that the set of knots~$\{K_n\,|\,n>1\}$ contains infinitely many pairwise not almost-concordant elements.
Consider the finite cover~$p \colon \wti Y \to Y$ associated to the kernel of
$\phi \colon \pi_1(Y) \to G$, and denote the covering link of $K_n$ by $L_n$, that is $L_n := p^{-1}(K_n)$.
As $\phi(\alpha)=0$, the restriction of $\phi$ to $\pi_1(B)$ is an epimorphism to the group generated by $\phi(\eta)$ in $G$. As $G$ is finite, and $\phi(\eta)\neq 0$, so $\phi(\eta)$ must generate a finite cyclic group $C_k$ for $k>1$. In other words, the cover of $B$ induced by $\phi$ is determined by an epimorphism $\pi_1(B)\twoheadrightarrow C_k$, and thus we obtain a cover which in each
component contains components of the link~$L_n$ as depicted in
Figure~\ref{fig:preimagehandlebody}.
Let $S_n := \{\lk(C, D) \,|\, C, D \text{ components of } L_n \} \subset \Z$ be the set of linking numbers.
\begin{figure}[h]
\labellist
\small\hair 0pt
\pinlabel $n$ [l] at 304 69
\pinlabel  \rotatebox{-90}{$n$} at 404 14
\pinlabel  \rotatebox{90}{$n$} at 215 14
\endlabellist
\includegraphics[width=0.75\textwidth]{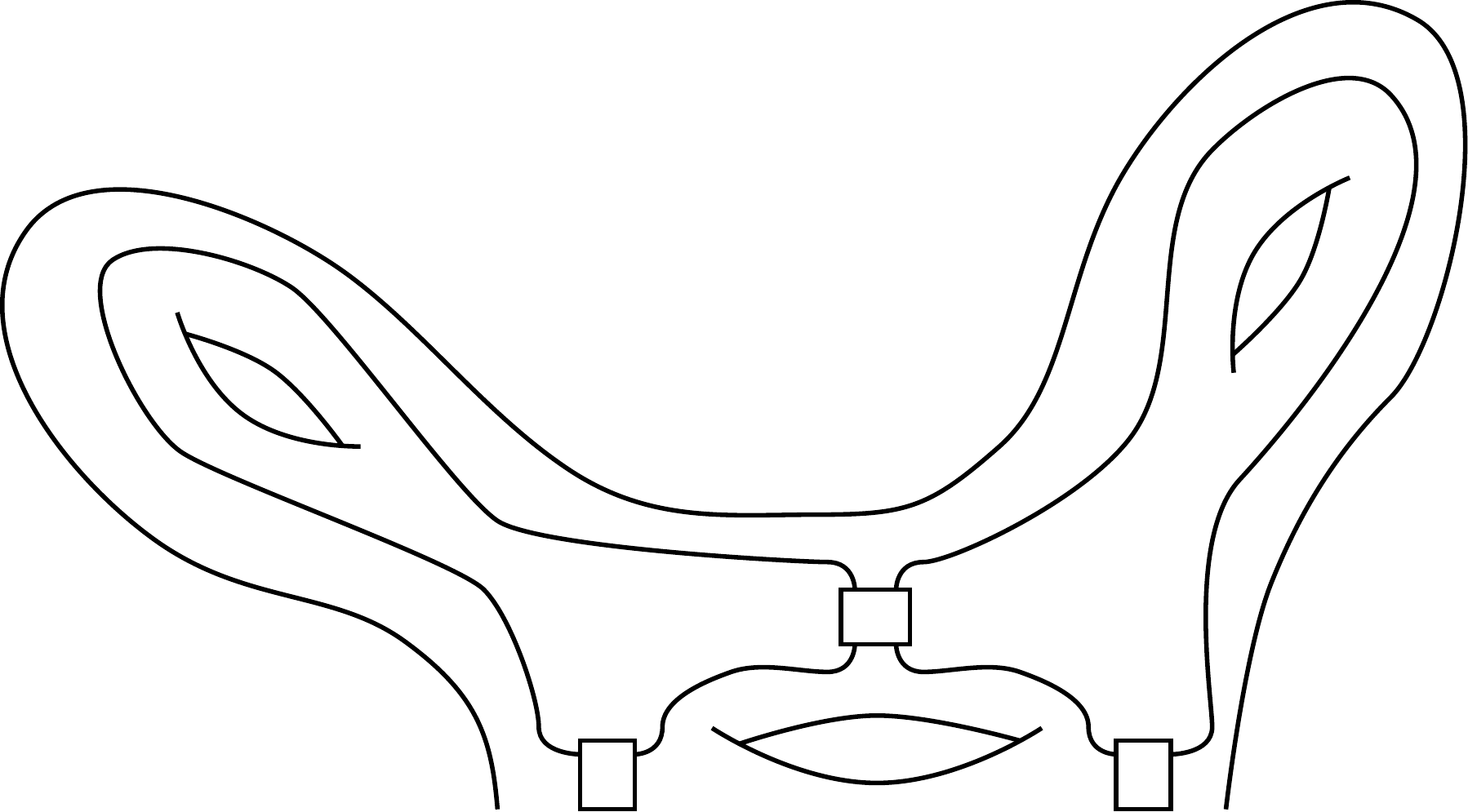}
\caption{Part of a component of the preimage of the handlebody~$B$, with some components of the link~$L_n$.}
\label{fig:preimagehandlebody}
\end{figure}

\begin{claim}
The maximal integer in the set $S_n$ becomes arbitrarily large as $n \rightarrow \infty$.
\end{claim}

Suppose that $\eta$ does not map to a $2$--torsion element in $C_k$. In the case that $\eta$ maps to $2$--torsion, the argument is similar, as in the proof of Proposition~\ref{prop:spherical-space-forms}.
Pick a connected component of the preimage of the handlebody~$B$ in~$\wti{Y}$.
Furthermore, pick two link components $C,D$ of $L_n$ in $B$,
which are related by an $n$--twist box and hence a single $n$--twist box as $\eta$ is not $2$--torsion in $C_k$.
Note that in the complement of $D$ in $Y$,
the homology class of $D$ decomposes as $[D] = [D_\text{dist}] + [D_\text{box}] \in H_1(\wti Y \sm \nu C; \Q)$,
where the curves $D_\text{dist}$ and $D_\text{box}$ are described in Figure~$\ref{fig:mickeysplit}$.
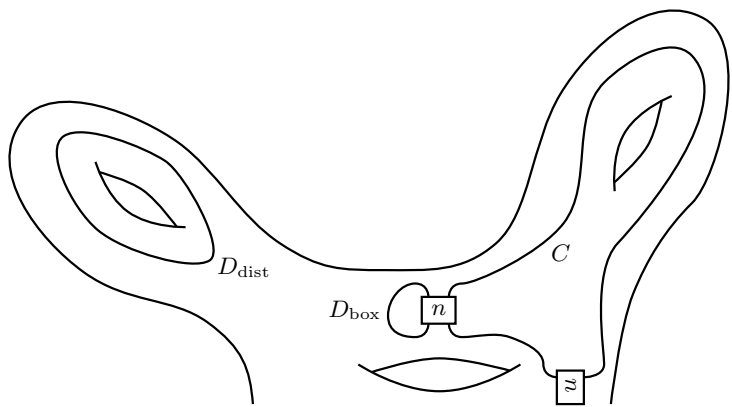
\begin{figure}[h]
\labellist
\small\hair 0pt
\pinlabel $n$ [l] at 304 69
\pinlabel $C$ [l] at 390 110
\pinlabel $D_\text{dist}$ [l] at 150 100
\pinlabel $D_\text{box}$ [l] at 230 69
\pinlabel  \rotatebox{-90}{$n$} at 404 14
\endlabellist
\includegraphics[width=0.75\textwidth]{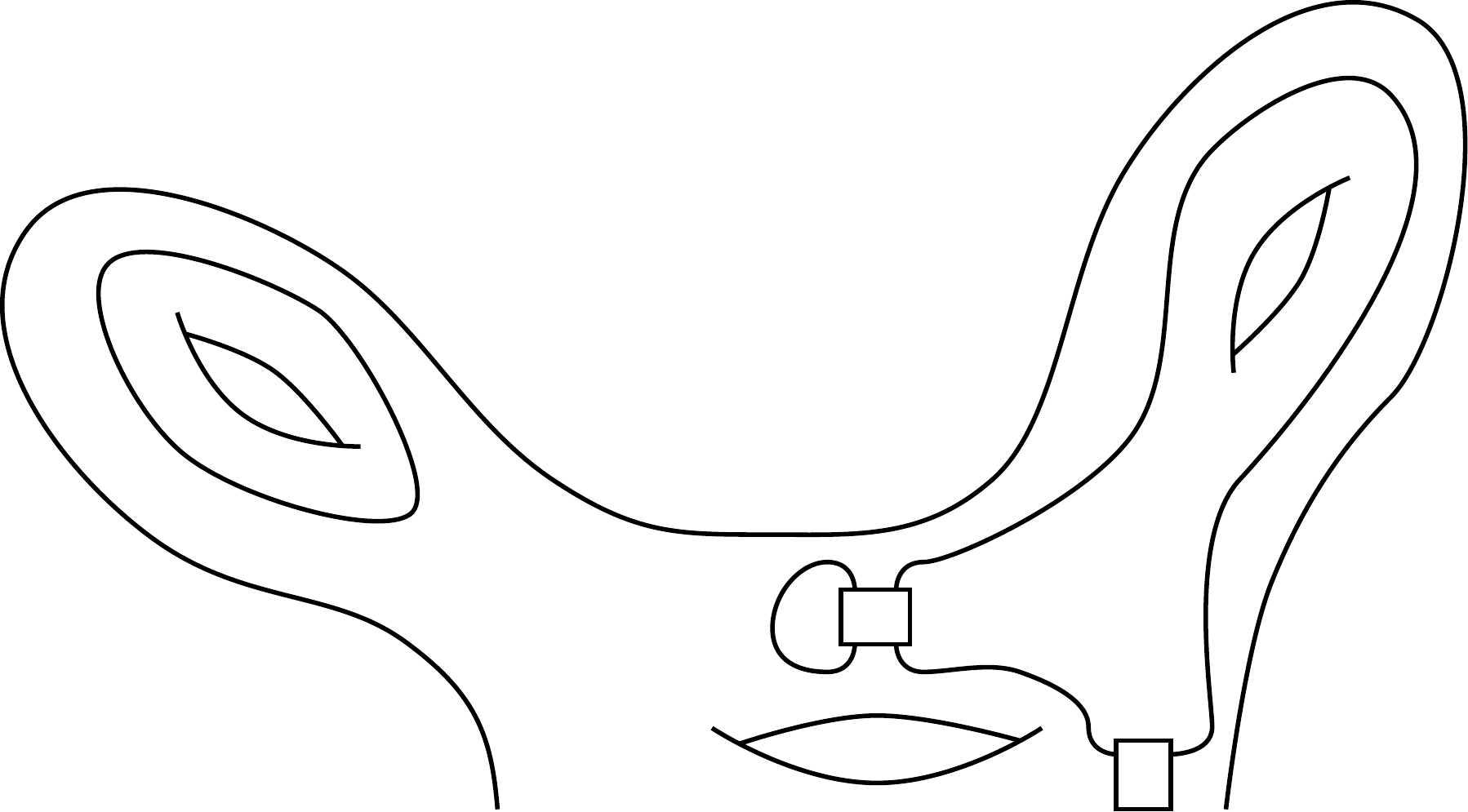}
\caption{The two contributions to the linking number~$\lk(C,D)$.}
\label{fig:mickeysplit}
\end{figure}
As $D_\text{box}$ is contained in a $3$--ball,
we may compute the linking number~$\lk(C, D_\text{box}) = n$. Consequently, we get
\[ \lk(C,D) = \lk(C, D_\text{dist}) + \lk(C, D_\text{box}) = \lk(C, D_\text{dist}) + n.\]
As the number~$\lk(C, D_\text{dist})$ is independent of $n$, this
proves the claim.

By Lemma~\ref{lem:linkingNumber}, the
set $S_n$ is an almost-concordance invariant and therefore the set $\{ L_n\,|\, n>1\}$ contains
infinitely many distinct almost-concordance classes. By Lemma~\ref{lem:CoveringLinks},
so does the set $\{ K_n \,|\, n>1 \}$.
\end{proof}

\begin{corollary}\label{cor:TorstionCase}$\,$
\begin{enumerate}[font=\normalfont]
\item Let $Y'$ be a spherical 3--manifold, and let $Y:= Y' \# Z$ for some $Z \neq S^3$.  Then any class in $x \in \pi_1(Y')$ is torsion, and since $\pi_1(Z) \neq 1$, we can apply the theorem to see that $\mathcal{C}_x(Y)$ contains infinitely many almost-concordance classes.
\item For any 3--manifold $Y \neq S^3$, the null-homotopic class $x$ contains infinitely many almost-concordance classes.
\end{enumerate}
\end{corollary}

The next theorem is not quite a corollary of Theorem~\ref{thm:linkingincover} because the class $x$ in question is not torsion in homotopy, however the same ideas as in that theorem also work in the following case.

\begin{theorem}\label{thm:non-separating-surface}
Let $Y$ be a closed oriented $3$--manifold. Suppose that $Y$ has a non-separating embedded oriented
surface~$\Sigma$, i.e.\ $\rk H_1(Y) \geq 1$.
Suppose $x=[\alpha]$ for a separating curve $\alpha$ on $\Sigma$.
Then $\mathcal{C}_x(Y)$ contains infinitely
many distinct almost-concordance classes.
\end{theorem}

\begin{proof}
For some $k > 1$, consider the cover~$\ol Y$ associated to the kernel~$\ker \phi$ of the map
\begin{align*}
\phi \colon \pi_1(Y) &\to \Z/k\Z\\
g &\mapsto [g] \cdot_Y [\Sigma]
\end{align*}
given by the intersection number, in $Y$, with the surface~$\Sigma$. Note that
$[\alpha] \in \ker \phi$. Furthermore, pick a surface $S \subset \Sigma$ bounding
$\alpha$. For this surface $S$, we have $\pi_1(S) \subset \ker \phi$ and so $S$ lifts along the cover $\ol Y \to Y$.

Pick a curve $\eta \subset Y$ not intersecting $S$ such that $\phi(\eta) = 1$, and pick a genus two handlebody~$B$ whose two cores are the curves $\eta$ and $\alpha$. Just as in Figure \ref{fig:nullhomotopypattern}, consider the knots~$K_n \subset B$ and also the covering links~$L_n$ of $K_n$, now corresponding to our cover $\ol Y \to Y$. The computation of linking numbers is in fact much easier now than in the previous theorem, as we can use $S$ to build a Seifert surface for $L_n$.
As shown in Figure \ref{fig:orsonspicture},\begin{figure}[h]
\labellist
\pinlabel $S$ [l] at 225 200
\pinlabel $\alpha$ [l] at 212 73
\pinlabel $\eta$ [l] at 500 77
\pinlabel $n$ [l] at 445 122
\endlabellist
\includegraphics[width=\textwidth]{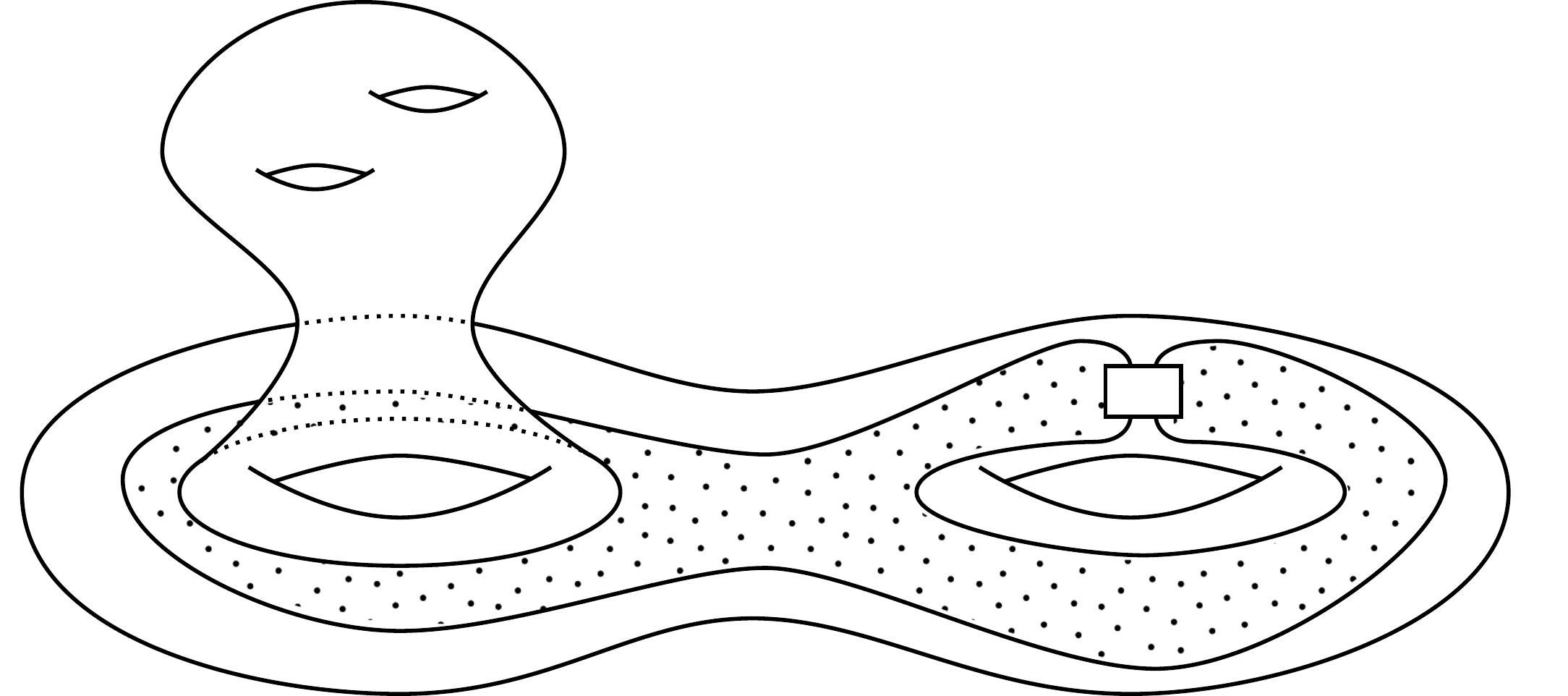}
\caption{The Seifert surface for $K_n$.}
\label{fig:orsonspicture}
\end{figure}
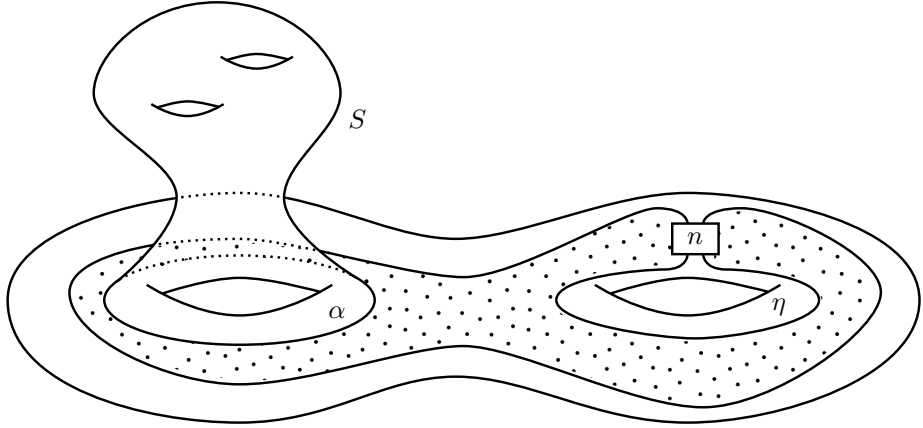 the link $\alpha\cup K_n$ has a Seifert surface in $B$, to which we attach the surface $S$ along $\alpha$, resulting in a Seifert surface for $K_n$. This surface clearly lifts to give a Seifert surface for each component of the covering link $L_n$.

Suppose $\phi(\eta)$ is not $2$--torsion. In the other case the argument is similar.
Take $C$, $D$ to be link components of $L_n$ related by an $n$--twist box, and
assume we have lifted our Seifert surface for $K_n$ to a Seifert surface for
$D$. Decomposing $D$ as in Figure~\ref{fig:mickeysplit}, we see that
$\lk(C,D)= n + \lk(C,D_\text{dist})$. But now the algebraic intersection $[{\ol S}]\cdot
[D_\text{dist}]=\lk(C,D_\text{dist})$ where ${\ol S}$ is the lift of $S$ in our lifted Seifert
surface. But $D_\text{dist}$ maps to the boundary of $S$ in $Y$, so we must have
geometric intersection ${\ol S}\cdot D_\text{dist}=0$.

Thus, any two components of $L_n$ link exactly 0 or $n$ times. By Lemma~\ref{lem:linkingNumber} we see that the links $L_n$ lie in distinct almost-concordance classes,
and so $\{ K_n \,|\, n>1\}$ represents a set of distinct almost-concordance classes in $\mathcal{C}_x(Y)$.
\end{proof}

\bibliographystyle{amsalpha}
\def\MR#1{}
\bibliography{writeup}

\providecommand{\bysame}{\leavevmode\hbox to3em{\hrulefill}\thinspace}
\providecommand{\MR}{\relax\ifhmode\unskip\space\fi MR }
\providecommand{\MRhref}[2]{%
  \href{http://www.ams.org/mathscinet-getitem?mr=#1}{#2}
}
\providecommand{\href}[2]{#2}
\begin{thebibliography}{{Cha}74}

\bibitem[AFW15]{MR3444187}
Matthias Aschenbrenner, Stefan Friedl, and Henry Wilton, \emph{3-manifold
  groups}, EMS Series of Lectures in Mathematics, European Mathematical Society
  (EMS), Z{\"u}rich, 2015. \MR{3444187}

\bibitem[Cel16]{Celoria:2016jk}
Daniele Celoria, \emph{On concordances in 3-manifolds}, Preprint, available at
  http://arxiv.org/abs/1602.05476v2, 2016.

\bibitem[CF08]{MR2357695}
David Cimasoni and Vincent Florens, \emph{Generalized {S}eifert surfaces and
  signatures of colored links}, Trans. Amer. Math. Soc. \textbf{360} (2008),
  no.~3, 1223--1264. \MR{2357695}

\bibitem[CF13]{MR3062861}
Jae~Choon Cha and Stefan Friedl, \emph{Twisted torsion invariants and link
  concordance}, Forum Math. \textbf{25} (2013), no.~3, 471--504. \MR{3062861}

\bibitem[{Cha}74]{MR0391109}
Thomas {Chapman}, \emph{Topological invariance of {W}hitehead torsion}, Amer.
  J. Math. \textbf{96} (1974), 488--497. \MR{0391109}

\bibitem[Coo82]{Cooper82}
Daryl Cooper, \emph{The universal abelian cover of a link}, Low-dimensional
  topology ({B}angor, 1979), London Math. Soc. Lecture Note Ser., vol.~48,
  Cambridge Univ. Press, Cambridge-New York, 1982, pp.~51--66. \MR{662427}

\bibitem[CP14]{MR3270170}
Jae~Choon Cha and Mark Powell, \emph{Nonconcordant links with homology
  cobordant zero-framed surgery manifolds}, Pacific J. Math. \textbf{272}
  (2014), no.~1, 1--33. \MR{3270170}

\bibitem[CS74]{MR0339216}
Sylvain~E. Cappell and Julius~L. Shaneson, \emph{The codimension two placement
  problem and homology equivalent manifolds}, Ann. of Math. (2) \textbf{99}
  (1974), 277--348. \MR{0339216 (49 \#3978)}

\bibitem[Dav06]{MR2212279}
James~F. Davis, \emph{A two component link with {A}lexander polynomial one is
  concordant to the {H}opf link}, Math. Proc. Cambridge Philos. Soc.
  \textbf{140} (2006), no.~2, 265--268. \MR{2212279 (2006k:57010)}

\bibitem[FK08]{FK08}
Stefan {Friedl} and Taehee {Kim}, \emph{{Twisted Alexander norms give lower
  bounds on the Thurston norm.}}, {Trans. Am. Math. Soc.} \textbf{360} (2008),
  no.~9, 4597--4618.

\bibitem[FQ90]{MR1201584}
Michael~H. Freedman and Frank Quinn, \emph{Topology of 4-manifolds}, Princeton
  Mathematical Series, vol.~39, Princeton University Press, Princeton, NJ,
  1990. \MR{1201584 (94b:57021)}

\bibitem[FV11]{MR2777847}
Stefan Friedl and Stefano Vidussi, \emph{A survey of twisted {A}lexander
  polynomials}, The mathematics of knots, Contrib. Math. Comput. Sci., vol.~1,
  Springer, Heidelberg, 2011, pp.~45--94. \MR{2777847}

\bibitem[Gol78]{MR521732}
Deborah~L. Goldsmith, \emph{A linking invariant of classical link concordance},
  Knot theory ({P}roc. {S}em., {P}lans-sur-{B}ex, 1977), Lecture Notes in
  Math., vol. 685, Springer, Berlin, 1978, pp.~135--170. \MR{521732}

\bibitem[Hec11]{Heck:2011}
Prudence Heck, \emph{Homotopy properties of knots in prime manifolds},
  Preprint, available at https://arxiv.org/abs/1110.6903, 2011.

\bibitem[Hil12]{MR2931688}
Jonathan Hillman, \emph{Algebraic invariants of links}, second ed., Series on
  Knots and Everything, vol.~52, World Scientific Publishing Co. Pte. Ltd.,
  Hackensack, NJ, 2012. \MR{2931688}

\bibitem[Lic97]{MR1472978}
W.~B.~Raymond Lickorish, \emph{An introduction to knot theory}, Graduate Texts
  in Mathematics, vol. 175, Springer-Verlag, New York, 1997. \MR{1472978}

\bibitem[Mil57]{Milnor:1957-1}
John~W. Milnor, \emph{Isotopy of links. {A}lgebraic geometry and topology}, A
  symposium in honor of S. Lefschetz, Princeton University Press, Princeton, N.
  J., 1957, pp.~280--306. \MR{19,1070c}

\bibitem[Mil95]{MR1354382}
David Miller, \emph{An extension of {M}ilnor's {$\overline \mu$}-invariants},
  Topology Appl. \textbf{65} (1995), no.~1, 69--82. \MR{1354382}

\bibitem[{Rol}85]{zbMATH03917275}
Dale {Rolfsen}, \emph{{Piecewise-linear I-equivalence of links.}}, {Low
  dimensional topology, 3rd Topology Semin. Univ. Sussex 1982, Lond. Math. Soc.
  Lect. Note Ser. 95, 161-178}, 1985.

\bibitem[Sch03]{MR2012959}
Rob Schneiderman, \emph{Algebraic linking numbers of knots in 3-manifolds},
  Algebr. Geom. Topol. \textbf{3} (2003), 921--968. \MR{2012959}

\bibitem[Ser03]{MR1954121}
Jean-Pierre Serre, \emph{Trees}, Springer Monographs in Mathematics,
  Springer-Verlag, Berlin, 2003, Corrected 2nd printing of the 1980 English
  translation. \MR{1954121}

\bibitem[{Sta}65]{zbMATH03218889}
John {Stallings}, \emph{{Homology and central series of groups.}}, {J. Algebra}
  \textbf{2} (1965), 170--181.

\bibitem[Tur01]{MR1809561}
Vladimir Turaev, \emph{Introduction to combinatorial torsions}, Lectures in
  Mathematics ETH Z{\"u}rich, Birkh{\"a}user Verlag, Basel, 2001. \MR{1809561}

\end{thebibliography}
\end{document}